\documentclass[a4paper]{amsart}
\usepackage{amssymb,latexsym}
\usepackage[utf8x]{inputenc}
\usepackage{amsmath,amsthm}
\usepackage{amsfonts,mathrsfs}
\usepackage{cleveref}
\usepackage{enumerate,units}
\usepackage{url}

\theoremstyle{plain}
\newtheorem{theorem}{Theorem}[section]
\newtheorem{corollary}[theorem]{Corollary}
\newtheorem{lemma}[theorem]{Lemma}
\newtheorem{proposition}[theorem]{Proposition}

\newtheorem{fact}[theorem]{Fact}
\newtheorem*{claim}{Claim}
\newtheorem*{theorem*}{Theorem A}

\theoremstyle{definition}
\newtheorem{definition}[theorem]{Definition}

\theoremstyle{remark}
\newtheorem{remark}{Remark}[section]

\numberwithin{equation}{section}
\newcommand{\forkindep}[1][]{%
  \mathrel{
    \mathop{
      \vcenter{
        \hbox{\oalign{\noalign{\kern-.3ex}\hfil$\vert$\hfil\cr
              \noalign{\kern-.7ex}
              $\smile$\cr\noalign{\kern-.3ex}}}
      }
    }\displaylimits_{#1}
  }
}

\newenvironment{claimproof}[1][\proofname]
  {%
    \proof[#1]%
  }
  {%
    \endproof%
  }

\newcounter{step}                   
    {\hfill $\clubsuit$             
     \vspace{7pt}\par}
  \newcommand\s[1][]{\if#1=1 \mathfrak{s}_p \else \mathfrak{s}_{p^{#1}}\fi}

\DeclareMathOperator{\acl}{acl}

\newcommand{\LPP}{\widetilde{\mathrm{PP}}}

\newcommand{\Pp}{{\mathcal P}}
\newcommand{\X}{{\mathcal X}}
\newcommand{\Y}{{\mathcal Y}}

\def\nsim{\raise.17ex\hbox{$\scriptstyle\sim$}}
\def\dprk{\mathrm{dp}}

\begin{document}
\title{The dp-rank of abelian groups}
\author{Yatir Halevi and Daniel Palac\'in}
\thanks{Both authors were partially supported by the  European Research Council grant 338821. The second author was also partially supported by MTM2017-86777-P}
\address{Einstein Institute of Mathematics, The Hebrew University of Jerusalem, Givat Ram 9190401, Jerusalem, Israel}
\email{yatir.halevi@mail.huji.ac.il}
\email{daniel.palacin@mail.huji.ac.il}

\begin{abstract}
An equation to compute the dp-rank of any abelian group is given. It is also shown that its dp-rank, or more generally that of any one-based group, agrees with its Vapnik–Chervonenkis density. Furthermore, strong abelian groups are characterised to be precisely those abelian groups $A$ such that there are only finitely many primes $p$ such that the group $A/pA$ is infinite and for every prime $p$, there are only finitely many natural numbers $n$ such that $(p^nA)[p]/(p^{n+1}A)[p]$ is infinite.

Finally, it is shown that an infinite stable field of finite \rm{dp}-rank is algebraically closed.
\end{abstract}

\maketitle

\section{Introduction}

A complete classification of abelian groups up to elementarily equivalence was provided by Szmielew \cite{Szm}, who determined a list of group-theoretic invariants that characterise the first-order theory of abelian groups. The purpose of this paper is, using the classification given by Szmielew, to compute the $\dprk$-rank of any abelian group. Furthermore, we show that it is uniquely determined by a suitable semilattice of its subgroups. 

The $\dprk$-rank is a model-theoretic rank for dependent theories, originating from Shelah's work on strongly dependent theories \cite{StrDep}, which bounds the diversity between realizations of a type. Originally, it was introduced as an analogue of weight in stable theories and it was then used to obtain a notion of minimality inside dependent theories. Since then, the $\dprk$-rank has been proven to play an important role, not only in the understanding of dependent theories but also in strengthening the connections between dependent theories and combinatorics. Of particular interest, is the relation between the $\dprk$-rank and the Vapnik–Chervonenkis density (vc-density). The vc-density has been studied for quite some time in combinatorics, computational geometry, statistics and machine learning. In model theory it relates to the classical problem of counting types and was furthermore studied extensively in many specific cases in \cite{ADHMS1,ADHMS}.

Concerning abelian groups, it was proven in \cite{ADHMS} that the vc-density is the least integer $d$, if it exists, such that the intersection of $n>d$ many definable subgroups  has finite index in a sub-intersection of $d$ many. In other words, the vc-density agrees with the breadth of the semilattice of commensurable classes of definable subgroups, see Section 2. Our first result is to show that this holds for any one-based group, a larger family of groups which includes abelian-by-finite groups as well as left modules over a unital ring. Furthermore, these two notions coincide with the $\dprk$-rank. This is Theorem \ref{T:DP-VC-Br}. As a consequence, we deduce that the vc-density of the theory of a one-based group has linear growth, answering a question from \cite{ADHMS}, see Corollary \ref{C:VC-Lin}.

We use this characterisation of the $\dprk$-rank to give an equation for the \rm{dp}-rank of abelian groups. For this, we recall the definition of the Szmielew invariants for an abelian group $A$, written additively. For a prime $p$ and a positive integer $n$, set
\begin{align*} 
\mathrm{U}(p, n; A) & = \left|( p^n A)[p] / (p^{n+1} A)[p] \right| , \\
\mathrm{D}(p, n; A) & = \left| (p^n A) [p] \right|, \\
{\rm Tf} (p, n; A) & = \left| p^n A / p^{n+1} A \right|, \\
{\rm Exp}(p, n; A) & = \left| p^n A \right| .
\end{align*}
The notation $\mathrm{U}$ stands for Ulm, who introduced this invariant, while $\mathrm{D}, {\rm Tf} $ and ${\rm Exp}$ stand for divisible, torsion-free and exponent respectively. It is clear that these invariants are in fact invariants of the theory of $A$. To better visualize these invariants, recall that every abelian group is elementary equivalent to a direct sum of abelian groups of the sort 
\[\mathbb{Z}(p^n),\, \mathbb{Z}_{(p)},\, \mathbb{Z}(p^\infty)\, \text{ and } \mathbb{Q},\]
where $p$ is a prime and $n$ is a natural number, by the work of Szmielew. A countable direct sum of these groups is called a Szmielew group, see Section \ref{ss:szmielew groups}.  For our purposes, we shall consider another closely related family of invariants of the theory. Following Eklof and Fischer \cite{EkFi}, define 
$$
\mathrm{D}(p;A) = \lim_{n\rightarrow\infty} \mathrm{D}(p, n; A) ,
$$
which determines the number of copies of the Pr\"ufer group $\mathbb Z(p^\infty)$ in any $\aleph_1$-saturated abelian group $A$ and 
$$
{\rm Tf}(p;A) = \lim_{n\rightarrow\infty} {\rm Tf} (p, n; A),
$$
which determines the number of copies of the torsion-free group $\mathbb Z_{(p)}$ in any $\aleph_1$-saturated abelian group $A$. 

Due to the infinitary behavour of the dp-rank, we are mainly interested in controlling when these groups appear infinitely many times. Thus, we denote by $\mathrm{D}_{\ge\aleph_0}(A)$ the set of primes $p$ such that $\mathrm{D}(p;A)$ is infinite and similarly $\mathrm{Tf}_{\ge\aleph_0}(A)$ denotes the set of primes $p$ that $\mathrm{Tf}(p;A)$ is infinite. 
Concerning the number of copies of the groups $\mathbb Z(p^n)$ in $A$, which are encoded by the Ulm invariant, using the notation from \cite{ADHMS} we write $\mathrm{U}_{\ge\aleph_0}(p;A)$ for the set of natural numbers $n$ such that $\mathrm{U}(p,n;A)$ is infinite. Furthermore, we denote by $\mathrm{U}_{\ge\aleph_0}(A)$ the set of primes $p$ such that $\mathrm{U}(p,n;A) >1$ for infinitely many $n$. In other words, the set $\mathrm{U}_{\ge \aleph_0}(A)$ is precisely the set of primes $p$ such that $A$ has unbounded $p$-length, {\it i.e.} there is no finite upper bound on the order of the elements of $p$-power torsion which are not divisible by $p$. 
It is clear that the sets $\mathrm{D}_{\ge\aleph_0}(A), \mathrm{Tf}_{\ge\aleph_0}(A) , \mathrm{U}_{\ge\aleph_0}(p;A)$ and $\mathrm{U}_{\ge\aleph_0}(A)$ are also invariants of the theory of $A$. 

As announced above, the main result of the paper is the following theorem in which we compute the $\dprk$-rank of any abelian group. For ease of presentation, we break the equation into cases. The finite exponent group case was essentially computed in \cite[Theorem 5.1]{ADHMS}. Following their notation, for a set $I$ of positive integers, we denote by $d(I)$ the size of the maximal subset $I_0$ of $I$ such that for every two distinct elements $i,j$ of $I_0$ we have $|j-i|\ge 2$.

\begin{theorem}\label{T:Main}
Let $A$ be an abelian group of finite $\dprk$-rank. Then one of the following holds:

\begin{enumerate}
\item $A$ is torsion-free and
$$
\dprk(A)= \max \left\{ 1, |\mathrm{Tf}_{\ge\aleph_0}(A)| \right\}.
$$
\item $A$ is of finite exponent and 
$$
\dprk(A)= \sum_p d \left(\mathrm U_{\ge\aleph_0}(p;A) \right).
$$ 

\item $A$ has unbounded exponent, torsion but has finite $p$-length for every prime $p$ and
$$
\dprk(A)=\sum_p d \left(\mathrm U_{\ge\aleph_0}(p;A) \right) + \max \left\{ 1, |\mathrm{Tf}_{\ge\aleph_0}(A)| , |\mathrm{D}_{\ge\aleph_0}(A)| \right\}.
$$ 
\item $A$ has unbounded $p$-length for every prime $p$ and
$$
\dprk(A)=\sum_p d \left(\mathrm U_{\ge\aleph_0}(p;A) \right) + |\mathrm{U}_{\ge\aleph_0}(A)| + \max \left\{|\mathrm{Tf}_{\ge\aleph_0}(A)| , |\mathrm{D}_{\ge\aleph_0}(A)| \right\}.
$$
\end{enumerate}
\end{theorem}

Of course, one can reformulate the statement above to get a single equation. For those that prefer to be concise, we write it bellow. Before proceeding we introduce some further notation. 

Set $\epsilon_\mathrm{U}$ equal to $0$ if the set $\mathrm{U}_{\ge\aleph_0}(A)$ is empty and $1$ otherwise, and $\epsilon_\mathrm{Exp}$ is equal to $0$ if $A$ has finite exponent and $1$ otherwise. Furthermore, define $\epsilon_\mathrm{Tf}$ to be $0$ if for every prime $p$ we have that $\mathrm{Tf}(p;A) = 0$ and $1$ in other case. Finally, define $\epsilon_\mathrm{D}$ likewise. Therefore, putting all this together, for any abelian group $A$ we obtain the following equation:
\begin{align*}
\dprk(A)  = & \ \sum_p d\big(\mathrm{U}_{\ge\aleph_0}(p;A)\big) + |\mathrm{U}_{\ge\aleph_0}(A)|  +  (1-\max\{\epsilon_{\mathrm{U}} , \epsilon_\mathrm{Tf} , \epsilon_\mathrm{D}\} )\cdot \epsilon_\mathrm{Exp}\\
& +  \max \{ \epsilon_\mathrm{Tf} , \epsilon_\mathrm{D} \} \cdot \max \left\{ 1 - \epsilon_{\mathrm{U}}, |{\rm Tf}_{\ge\aleph_0}(A)|,  |{\rm D}_{\ge\aleph_0}(A)| \right\}.
\end{align*}

The proof of the main theorem is preceded, in Section \ref{s:cases}, by a study of the different components building a Szmielew group: we give equations for their \rm{dp}-rank and characterize the inp-patterns needed to witness that \rm{dp}-rank.
In Section \ref{s:char-strong} we characterize strong abelian groups and abelian groups of finite \rm{dp}-rank. The latter was already obtained in \cite[Theorem 5.23]{ADHMS}. As for the former, we show that an abelian group is strong if and only if there is only a finite number of primes $p$ such that the group $A/pA$ is infinite and only a finite number of pairs $(p,n)$ such that the invariant $\mathrm U(p,n;A)$ is infinite. Finally, Section \ref{s:main} is devoted to proving the main theorem. 

In Section \ref{s:last}, the last section of the paper, we point out that the behaviour of the $\dprk$-rank is completely different whenever the given group carries some additional structure. In particular, we give an example of a divisible torsion-free abelian group with additional structure whose theory is $\omega$-stable but does not have finite \rm{dp}-rank. Furthermore, and in the spirit of abelian groups with additional structure, we end the paper by showing that infinite stable fields of finite \rm{dp}-rank are algebraically closed.

Before concluding the introduction, we would like to remark that Proposition \ref{P:inpattern-in-1based} was used by the first author and Assaf Hasson to describe the dp-rank of ordered abelian groups \cite{HaH1}. We thank him for going over a first version of the proof of that proposition. We also thank the referees for several helpful comments and suggestions, which improved the final version.

\section{Preliminaries on the theory of abelian groups}

In this section we recall the basics of the first-order theory of abelian groups needed for the paper. For a detailed exposition we refer to \cite[Appendix A.2]{Hodges}, which we also use as a main reference. 

Before proceeding, we discuss some notational conventions. We write abelian groups additively. If $A$ is an abelian group, then the subgroup of elements of order $n$ is denoted by $A[n]$, which is of course defined by the formula $nx=0$. As usual, the subgroup of $n$-powers of $A$ is denoted by $nA$.

\subsection{Quantifier elimination} We construe abelian groups as structures in the language of abelian groups or interchangeably in the language of $\mathbb Z$-modules, since both languages are definitionally equivalent. Thus, we can use the Baur-Monk quantifier elimination given in \cite{Baur,Monk} for left modules over unital rings to describe definable sets in abelian groups, see \cite[Corollary A.1.2]{Hodges}. Namely, any formula is equivalent to a boolean combination of positive primitive formulas; recall that a formula $\varphi(x_1,\ldots,x_m)$ is {\em positive primitive}, p.p. for short, if it is equivalent to a formula of the form 
$$
\exists  y_1\ldots y_n \bigwedge_k \left(\sum_{i=1}^{m} \lambda_i x_i + \sum_{j=1}^{n} \mu_j y_j = 0\right)
$$
for some $k\in \mathbb N$ and $\lambda_i,\mu_j\in\mathbb Z$. Given an abelian group $A$, it is easy to see that this formula defines an $\emptyset$-definable subgroup of $A^{m}$. Furthermore, we have an explicit description of p.p.-definable subgroups of an abelian group, see for instance \cite[Lemma A.2.1]{Hodges}.

\begin{fact}[Pr\"ufer]\label{F:prufer}
Every p.p.-definable proper subgroup of an abelian group is given by a finite conjunction of formulas of the form $mx = 0$ with $0 \le m$ and $p^n | p^m x$ with $0 \le m < n$.
\end{fact}

\begin{remark}\label{R:p^n|p^mx}
The formula $p^n|p^mx$ with $0\le m<n$ defines in an abelian group $A$ the subgroup $p^{n-m}A+A[p^m]$.
\end{remark}

Using the Baur-Monk quantifier elimination it is easy to see that any abelian group $A$ has a {\em stable theory}, {\it i.e.} for every formula $\phi(\bar x;\bar y)$, in the language of groups, there is some $k$ such that there is no sequence $(\bar a_i,\bar b_i)_{i\le k}$ in $A$ for which $\phi(\bar a_i;\bar b_j)$ holds iff $i\le j$. In fact, other related notions can be also explained via the structure of the lattice of its p.p.-definable subgroups. 
Namely, an abelian group is {\em superstable} if and only if there is no infinite descending chain of p.p.-definable subgroups each of infinite index in its predecessor, and it is {\em $\omega$-stable} if and only if it satisfies the minimal chain condition on p.p.-definable subgroups. 

Finally we will use the following easy observation, implicitly, throughout the paper:
\begin{fact}\label{F:pp-commutes-with-sum}
Let $\varphi(x)$ define a p.p.-definable subgroup of an abelian group $A$. If $A=B\oplus C$ then $\varphi (A)=\varphi (B)\oplus \varphi (C)$.
\end{fact}

\subsection{Commensurability and p.p.-definable subgroups} Given a group $G$, we say that a subgroup $H$ of $G$ is contained up to finite index in another subgroup $N$ of $G$ if $H\cap N$ has finite index in $H$. Then $H$ and $N$ are {\em commensurable}, if $H$ and $N$ are contained up to finite index in $N$ and $H$ respectively. Equivalently, if $H\cap N$ has finite index in both $H$ and $N$. It is clear that containment up to finite index is a transitive relation among subgroups of $G$, and that commensurability is an equivalence relation. 

In the case of an abelian group $A$, these notions induce a natural join-semilattice structure on the set of commensurability classes of the p.p.-definable subgroups of $A$. Following the notation from \cite[Section 4.6]{ADHMS} we denote this semilattice by $\LPP(A)$, also see \cite{ADHMS} for more information. 

\begin{definition} The {\em breadth} of  a join-semilattice $L$ is the smallest integer $d$, if it exists, such that for all $x_1,\ldots,x_n$ in $L$ with $n>d$ there are positive integers $1\le i_1<\ldots<i_d \le n$ such that $x_1\wedge \dots \wedge x_n = x_{i_1}\wedge \dots \wedge x_{i_d}$. We denote by ${\rm breadth}(L)$ the integer $d$, if it exists.
\end{definition}
If $L$ and $L'$ are two distinct join-semilattice, we easily have 
$$
{\rm breadth}(L\times L') \le {\rm breadth}(L) + {\rm breadth}(L'),
$$
with equality if both $L$ and $L'$ have a greatest and a smallest element. In our situation, note that the semilattice $\LPP(A)$ of an abelian group $A$ has a greatest and smallest element. Moreover, it has breadth $d$ if and only if any finite intersection of subgroups of $A$ has finite index in some sub-intersection of $d$ many subgroups. 

Now, given two abelian groups $A_1$ and $A_2$, each p.p.-definable subgroup $H$ of $A_1 \oplus A_2$ can obviously be written as $(H\cap A_1) \oplus (H\cap A_2)$. Thus, we obtain a natural embedding  of semilattices from $\LPP(A_1\oplus A_2)$ into $\LPP(A_1)\times \LPP(A_2)$, which is a surjection when each $A_i$ is p.p.-definable in $A_1\oplus A_2$.  Therefore, we get the following fact, which corresponds to \cite[Lemma 4.18]{ADHMS}.

\begin{fact}\label{F:sum-of-breadth}
If $A_1$ and $A_2$ are two abelian groups, then
$$
\max_{i=1,2} \left\{ {\rm breadth} \left(\LPP(A_i) \right) \right\} \le {\rm breadth}\left(\LPP(A_1 \oplus A_2)\right) \le  \sum_{i=1,2}{\rm breadth} \left(\LPP(A_i)\right),
$$ with equality in the second inequation if both $A_1$ and $A_2$ are p.p.-definable in $A_1\oplus A_2$. 
\end{fact}

In Section \ref{s:dprk,vcden}, we shall prove that the breadth of any abelian group, or more generally any one-based group, is precisely its $\dprk$-rank and consequently, obtain these properties for the dp-rank of abelian groups. 

\subsection{Szmielew groups}\label{ss:szmielew groups} Let $A$ be an abelian group and recall that we write it additively. We put $A_p$ for the {\em primary $p$-component} of the torsion subgroup of $A$. That is to say, $A_p$ consists of all elements of $A$ whose orders are powers of the prime $p$. Thus, the torsion subgroup is the direct sum of all the $A_p$ for $p$ prime. We denote by $A^{(\alpha)}$ the direct sum of $\alpha$ copies of $A$. 

An abelian group has {\em finite exponent} if there exists some positive integer $n$ such that $nA=\{0\}$; otherwise $A$ has {\em unbounded exponent}. For a prime $p$, we say that $A$ has {\em finite $p$-length} if there is a finite upper bound on the order of the elements of $A_p$ which are not divisible by $p$. Otherwise, we say that $A$ has unbounded $p$-length. Notice that $A$ has unbounded $p$-length if and only if so does any elementary equivalent group.

\begin{definition} A \emph{Szmielew group} is a countable abelian group of the form
$$
A=\bigoplus_p \left( \bigoplus_{n>0} \mathbb{Z}(p^n)^{(\alpha_{p,n})}\oplus (\mathbb{Z}_{(p)})^{(\beta_p)} \oplus \mathbb{Z}(p^\infty)^{(\gamma_p)}\right) \oplus \mathbb{Q}^{(\delta)},
$$
where each $\alpha_{p,n}, \beta_p, \gamma_p$ and $\delta$ are finite or countably infinite, $\mathbb Z(p^n)$ is the cyclic group of order $p^n$, $\mathbb{Z}_{(p)}$ is the additive group of all rational numbers with denominator not divisible by $p$, and $\mathbb{Z}(p^\infty)$ is the Pr\"ufer $p$-group. Such a Szmielew group is called {\em strict} if the following holds:
\begin{itemize}
\item $\delta$ is $0$ or $\omega$;
\item if either $\alpha_{p,n}\neq 0$ for infinitely many pairs $(p,n)$ or $\beta_p\neq 0$ or $\gamma_p\neq 0$ for some $p$, then $\delta=0$; and
\item for each prime $p$, if $\alpha_{p,n}\neq 0$ for infinite many $n$, then $\beta_p = \gamma_p = 0$.
\end{itemize}
Notice that the second bullet point asserts that if the part preceding `$\oplus \mathbb Q^{(\delta)}$' has unbounded exponent, then $\delta=0$.
\end{definition}
If $A$ is a Szmielew group, we have that $A$ has unbounded $p$-length if and only if $\alpha_{p,n}\neq 0$ for infinitely many $n$. Using the notation from the introduction we then have that:
$$
\mathrm{Tf}(p;A) = p^{\beta_p}, \ \mathrm{D}(p;A) = p^{\gamma_p} \ \text{ and } \ \mathrm{U}(p,n;A) = p^{\alpha_{p,n}}.
$$
Consequently, the sets $\mathrm{Tf}_{\ge \aleph_0}(A)$ and $\mathrm{D}_{\ge\aleph_0}(A)$ denote the collection of primes $p$ such that $\beta_p=\omega$ and $\gamma_p=\omega$ respectively, and similarly $\mathrm{U}_{\ge \aleph_0}(p;A)$ is the set of all $n$ such that $\alpha_{p,n}=\omega$. Furthermore, observe that $\mathrm{U}_{\ge\aleph_0}(A)$ corresponds to the set of primes $p$ such that $\alpha_{p,n}\neq 0$ for infinitely many $n$. That is, the set $\mathrm{U}_{\ge\aleph_0}(A)$ is the collection of primes $p$ such that $A$ has unbounded $p$-length. Since all these sets are invariant under elementarily equivalence, the following seminal result of Szmielew \cite{Szm} allows us to restrict our attention to  strict Szmielew groups for our purposes; we refer the reader to \cite[Appendix A.2]{Hodges} for a proof.
\begin{fact}
Any abelian group is elementarily equivalent to a unique strict Szmielew group, up to isomorphism.
\end{fact}

We finish this section with an easy observation.

\begin{lemma}\label{L:p.p.-non-trivial}
If $A$ is an abelian group of unbounded exponent, then any p.p. formula of the form $\bigwedge_k n_k |m_kx$ with $0\le m_k< n_k$ defines an infinite subgroup of unbounded exponent.
\end{lemma} 
\begin{proof}
Let $A$ be an infinite abelian group, which we may assume to be a strict Szmielew group. By strictness, it suffices to see that the given p.p-formula $\varphi(x)$ defines a non-trivial subgroup in $\bigoplus_{n>0} \mathbb Z(p^n)^{(\alpha_n)}$ with $\alpha_n\neq 0$ for infinitely many $n$, in $\mathbb Z_{(p)}$, in $\mathbb Z(p^\infty)$ and also in $\mathbb Q$. The last three cases are clear. For the first case, note that this formula is equivalent, in $\bigoplus_{n>0} \mathbb Z(p^n)^{(\alpha_n)}$, to $\bigwedge_k (p^{r_k}|p^{s_k}x)$ with $0\le s_k<r_k$. Setting $n = \max_i \{r_i-s_i\}$, one can easily see that $\varphi(x)$ defines the subgroup $p^n \mathbb Z(p^k)$ of $\mathbb Z(p^k)$ for $k\ge \max_i\{r_i\}$ and hence, the result follows. 
\end{proof}

\section{dp-rank and vc-density}\label{s:dprk,vcden}
In this section, we prove that for an abelian group $A$ the breadth of $\LPP(A)$ coincides with the $\dprk$-rank and the {\rm vc}-density of the theory of $A$; the equivalence between breadth and {\rm vc}-density was proved in \cite[Corollary 4.13]{ADHMS}. In fact, we shall prove this for a larger class of groups called one-based groups, generalizing the aforementioned result, as well as some other results from \cite[Section 4.5]{ADHMS}. 

Recall that a group $G$ is {\em one-based} if any definable subset of $G^m$ is a finite boolean combination of cosets of $\acl^{\rm eq}(\emptyset)$-definable subgroups of $G^m$. Thus, by the Baur-Monk quantifier elimination, left modules over a unital ring are one-based groups and in fact any one-based group is abelian-by-finite by \cite[Theorem 3.2]{HruPil}. On the other hand, a one-based group can have additional structure other than the one coming from the group (or module) language. For instance, the class of one-based groups contains {\em abelian structures}, {\it i.e.} an abelian group $A$ with some predicates for subgroups of Cartesian powers of $A$ is one-based, see \cite[Theorem 4.2.8]{WagBook}. Similarly as in the abelian case, using the description of definable sets one can see that any one-based group has a stable theory.

\subsection{Dp-rank} For convenience, we fix a complete theory with infinite models in a given language. We shall be working inside a sufficiently saturated model of the theory. Thus, tuples and sets are assumed to be taken in this ambient model. We begin by recalling the following combinatorial pattern:

\begin{definition}
Let $\kappa$ be a cardinal. An \emph{inp-pattern of depth $\kappa$} for a partial type $\pi(x)$ is a sequence $(\varphi^\alpha(x,y^\alpha))_{\alpha<\kappa}$ of formulas together with a sequence of natural numbers $(k^\alpha)_{\alpha<\kappa}$ such that there is an array of tuples $b_i^\alpha$ for $\alpha<\kappa$ and $i<\omega$ for which each `row' $\{\varphi^{\alpha}(x,b_i^{\alpha})\}_{i<\omega}$ is $k^{\alpha}$-inconsistent, but every `path' $\pi(x)\cup \{\varphi^{\alpha}(x,b_{\eta(\alpha)}^{\alpha})\}_{\alpha<\kappa}$ for a function $\eta:\kappa\rightarrow \omega$ is consistent.
\end{definition}

\begin{remark}\label{R:MutInd} In the definition of {\rm inp}-pattern, the sequences $(b_i^\alpha)_{i<\omega}$ with $\alpha<\kappa$ can be taken to be {\em mutually indiscernible}, {\it i.e.} each $(b_i^\alpha)_{i<\omega}$ is indiscernible over $\{b_i^\beta\}_{i<\omega,\beta\neq\alpha}$, see \cite[Proposition 6]{Adler} or \cite[Lemma 4.2]{SimBook}.
\end{remark}

To avoid an excess of generality, we only recall the characterisation of {\rm dp}-rank for stable theories. In particular, we relate it to the notion of weight. Recall that the {\em weight} of a complete type $p$ is the supremum of the set of all cardinalities $\kappa$ for which there exists a non-forking extension $q\in S(M)$ of $p$ and an $M$-independent sequence $(b_i)_{i<\kappa}$ such that $\mathrm{tp}(a/M,b_i)$ forks over $M$. 

We refer to \cite{Adler} for basic facts about the {\rm dp}-rank and related notions, in particular to the following result:

\begin{fact} In a stable theory, the {\em {\rm dp}-rank} of a partial type $\pi$ equals the supremum of the cardinalities $\kappa$ of all possible {\rm inp}-patterns for $\pi$ and moreover, it is the supremum of the weights of all complete types extending $\pi$. 

\end{fact}

The {\rm dp}-rank of a stable theory (or structure) is the {\rm dp}-rank of the partial type given by the formula $x=x$.  
We say that a stable theory is {\em {\rm dp}-minimal} if the partial type $\{x=x\}$ only admits {\rm inp}-patterns of depth $1$, and {\em strong} if there is no infinite {\rm inp}-pattern or equivalently if every type has finite weight. Let us remark that the {\rm dp}-rank of a stable countable theory is at most $\aleph_0$, but such a theory may not be strong. 

The following result illustrates the relation between inp-patterns and the semilattice of commensurability classes of definable subgroups.

\begin{proposition}\label{P:inpattern-in-1based}
A one-based group admits an inp-pattern of depth $\kappa$ if and only if there exist $\acl^{\rm eq}(\emptyset)$-definable subgroups $(H_\alpha)_{\alpha<\kappa}$ such that for any $i_0<\kappa$ the index 
\[
\left[\bigcap_{\alpha\neq i_0} H_\alpha:\bigcap_{\alpha}H_\alpha\right]=\infty.
\]
Furthermore, in the latter case, these subgroups $H_\alpha$ witness an inp-pattern of depth $\kappa$, {\it i.e.} there exists an indiscernible array $(b^\alpha_i)_{\alpha<\kappa,i<\omega}$, such that $\{x\in b_i^\alpha H_\alpha\}_{\alpha<\kappa,i<\omega}$ forms an inp-pattern of depth $\kappa$.
\end{proposition}
\begin{remark}
In the above description, we use the convention that the empty intersection of subgroups is the whole group. Thus an inp-pattern of depth $1$ will be witnessed by any $\acl^{\rm eq}(\emptyset)$-definable infinite index subgroup.
\end{remark}
\begin{proof}
Let $G$ be a one-based group and suppose that it admits an {\rm inp}-pattern $(\varphi^\alpha(x,y^\alpha))_{\alpha<\kappa}$ witnessed by natural numbers $(k^\alpha)_{\alpha<\kappa}$ and an array of tuples $(b_i^\alpha)_{i<\omega,\alpha<\kappa}$, where each $b_i^\alpha= (b_{i,0}^\alpha,\ldots,b_{i,n_\alpha}^\alpha)$. 

We start with some reductions. By one-basedness, we may assume that each formula $\varphi^{\alpha}(x,b_i^{\alpha})$ is indeed a finite boolean combination of cosets of $\acl^{\rm eq}(\emptyset)$-definable subgroups and in addition, by Remark \ref{R:MutInd}, we may take the sequences $(b_i^\alpha)_{i<\omega}$ to be mutually indiscernible. Thus, as the sequence $(b_i^{\alpha})_{i<\omega}$ is $\acl^{\rm eq}(\emptyset)$-indiscernible, the definable subgroups determining the cosets appearing in $\varphi^\alpha(x,b_i^{\alpha})$ are the same as the ones giving the cosets appearing in $\varphi^{\alpha}(x,b_0^{\alpha})$. 

Furthermore, by mutually indiscernibility and the consistency of paths we reduce to the case where all the $\varphi^{\alpha}(x,b_i^{\alpha})$ are without disjunctions. Therefore, for every $\alpha$ the formula $\varphi^{\alpha}(x,b_i^{\alpha})$ may be assumed to be of the form
\[ 
x\in b_{i,0}^{\alpha} \, H^{\alpha}_0\wedge x\notin b_{i,1}^{\alpha} \, H^{\alpha}_1\wedge \dots \wedge x\notin b_{i,n_\alpha}^{\alpha}\,  H^{\alpha}_{n_\alpha}
,\]
where $H_0^\alpha, \dots, H_{n_\alpha}^\alpha$ are  $\acl^{\rm eq}(\emptyset)$-definable subgroups. 


\begin{claim}
For each $\alpha$, the cosets $b_{i,0}^\alpha\, H_0^{\alpha}$ are all distinct.
\end{claim}
\begin{claimproof}
Otherwise, by indiscernibility all are the same and so, possibly after multiplying on the left by $(b_{i,0}^{\alpha})^{-1}$, we may suppose that $b_{i,0}^\alpha=1$ for every $i$. Let $a$ be an element of $G$ satisfying the formula $\varphi^\alpha(x,b_{0}^\alpha)$, {\it i.e.} 
$$
a\in H_0^\alpha \setminus \bigcup_{j=1}^{n_\alpha} b_{0,j}^\alpha \, H_j^\alpha.
$$ 
Since the set $\left \{\varphi^\alpha(x,b_i^\alpha) \right \}_{i<\omega}$ is $k^\alpha$-inconsistent, we can clearly find  infinitely many indices $i$ such that $a$ belongs to $\bigcup_{j=1}^{n_\alpha} b_{i,j}^\alpha H_j^\alpha$. By the pigeonhole principle, in fact, there is some index $k$ such that $a$ belongs to $b_{i,k}^\alpha H_k^\alpha$ for infinitely many indices $i$. Thus, all these cosets $b_{i,k}^\alpha H_k^\alpha$ are equal and hence $b_{0,k}^\alpha H_k^\alpha = b_{i,k}^\alpha H_k^\alpha$ by indiscernibility, contradicting the choice of the element $a$.
\end{claimproof}

It then follows by the Claim that the set of formulas $x\in b_{i,0}^{\alpha} \, H_0^\alpha$ for $i<\omega$ and $\alpha<\kappa$ yield an {\rm inp}-pattern of depth $\kappa$, as desired. We may, therefore, drop the index $0$ and simply write $b_i^\alpha$ and $H^\alpha$ instead of $b_{i,0}^\alpha$ and $H_0^\alpha$ respectively.

It remains to show that the the intersection of all subgroups $H^\alpha$ has unbounded index in the intersection of all except one. If $\kappa=1$, there is nothing to show. Otherwise, fix some $\beta<\kappa$. By assumption, for each $i$ the `path' \[\{\varphi^{\alpha}(x,b_{0}^{\alpha})\}_{\alpha\neq \beta}\cup\{\varphi_{\beta}(x,b_{i}^{\beta})\}\] is consistent and so, for every $i$ we can find an element $a_i$ in $\bigcap_{\alpha\neq \beta} b_0^\alpha H^\alpha\cap b_i^{\beta}H^\beta$. Note that $a_iH^\alpha = b_0^\alpha H^\alpha$ for every $i<\omega$ and $\alpha\neq\beta$; thus, every $a_0^{-1}a_i$ belongs to the intersection $\bigcap_{\alpha\neq\beta} H^\alpha$. 
Since the cosets $a_iH^\beta= b_i^\beta H^\beta$ are all distinct for $i<\omega$ by the Claim,  the cosets 
\[
 a_0^{-1}a_i \left(\bigcap_{\alpha\neq\beta} H^\alpha \cap  H^\beta \right) = \bigcap_{\alpha\neq\beta} H^\alpha \cap a_0^{-1}a_i H^\beta 
\]
are also distinct. Hence, by compactness the subgroup $\bigcap_{\alpha<\kappa} H^\alpha$ has unbounded index in $\bigcap_{\alpha\neq\beta} H^\alpha$, yielding the result.

For the other direction, suppose that $G$ is one-based and  that there are $\acl^{\rm eq}(\emptyset)$-definable subgroups $H_\alpha$ for $\alpha<\kappa$ such that the intersection of all of them have unbounded index in any proper sub-intersection. Since $G$ is abelian-by-finite, there exists a maximal abelian normal subgroup $A$ of $G$ of finite index. Note that $A$ is $\emptyset$-definable, since  by maximality it is $\emptyset$-invariant and equals the definable subgroup $Z(C_G(A))$. For each $\alpha$ we have
\[H_\alpha \slash (H_\alpha\cap A)\cong H_\alpha A\slash A
\] 
and so the subgroup $H_\alpha \cap A$ has finite index in $H_\alpha$. Replacing each $H_\alpha$ by this finite index subgroup, we may assume that every subgroup $H_\alpha$ is a subgroup of $A$. Hence, the same proof as in \cite[Proposition 4.5]{cher-kap-sim} gives an inp-pattern of depth $\kappa$ for $A$ and so for $G$. This finishes the proof.
\end{proof}

\begin{remark}\label{R:Inp-pattern}
An inspection of the proof yields that if any definable set is a boolean combination of cosets of definable subgroups from a family $\mathfrak F$ which is closed under finite intersections, then an inp-pattern can be witnessed by groups from $\mathfrak F$. In particular, for abelian groups we can take $\mathfrak F$ to be the family of p.p.-definable subgroups.

In fact, it is not necessary for $\mathfrak{F}$ to be closed under finite intersections. For if we have an inp-pattern of the form $\{x\in b_i^\alpha H_\alpha\}_{\alpha<\kappa,i<\omega}$, as in the proposition, such that each $H_\alpha$ is the intersection of some family of groups $\{H_{\alpha,j}\}_{j<n_\alpha}$ from $\mathfrak F$, then there must exist some $j_\alpha<n_\alpha$ such that $\{b_i^\alpha H_{\alpha,j_\alpha}\}_{i<\omega}$ is also 2-inconsistent. Consequently, we get that $\{x\in b_i^\alpha H_{\alpha,j_\alpha}\}_{\alpha<\kappa,i<\omega}$ forms an inp-pattern of depth $\kappa$. In particular, for abelian groups we can take $\mathfrak{F}$ to be the family of subgroups defined by formulas of the form $mx = 0$ with $0 \le m$ or $p^n | p^m x$ with $0 \le m < n$.
\end{remark}

Combining the previous proposition and remark with the results from Section 2.2 we obtain:

\begin{corollary}\label{C:breadth=dp}
Let $G$ be a one-based group. The $\dprk$-rank of $G$ is finite if and only if the breadth of the semilattice of commensurability classes of the $\acl^{\rm eq}(\emptyset)$-definable subgroups of $G$ exists. Furthermore, if this happens then they are equal.
\end{corollary}
\begin{proof}
Observe first that, by definition, given two $\acl^{\rm eq}(\emptyset)$-definable subgroups $H$ and $N$ of $G$, with commensurable classes $x$ and $y$ respectively, we have that
\[
\left[ H : H\cap N \right] <\omega \quad \Leftrightarrow \quad x = x\wedge y .
\] 
Assume now that $d=\dprk(G)$ and let $H_0,\ldots, H_n$ be $\acl^{\rm eq}(\emptyset)$-definable subgroups of $G$ with $n>d$. Then, applying several times Proposition \ref{P:inpattern-in-1based}, we can find indices $0\le i_1<\ldots<i_d\le n$ such that
$$
\left[ \bigcap_{i_j} H_{i_j} : \bigcap_{i} H_{i} \right] < \omega.
$$
Hence, if the $\dprk(G)$ is finite, then the breadth is at most $\dprk(G)$ by the minimality. To show the other direction, assume that the breadth of the semilattice of commensurability classes exists and is equal to $m$. Aiming for a contradiction assume that $m<\dprk(G)$. We may thus find definable subgroups $\{H_i\}_{i<m+1}$ witnessing an inp-pattern of depth $m+1$. However, by the definition of the breadth 
\[
\left[\bigcap_{i\neq i_0} H_i:\bigcap_{i<m+1} H_i\right]<\infty
\] 
for some $i_0<m+1$, a contradiction. Therefore, the group $G$ has finite dp-rank bounded above by $m$.
%
\end{proof}

\begin{corollary}\label{C:BasicEquation}
Let $A$ and $B$ are two abelian groups. If both are strong, then so is $A\oplus B$. Moreover, we have:
$$
\max\{\dprk(A),\dprk(B)\} \le \dprk(A \oplus B) \le \dprk(A) + \dprk(B),
$$ with equality in the second inequation if both $A$ and $B$ are p.p.-definable in $A\oplus B$. 
\end{corollary}
\begin{proof}
The moreover part is a consequence of Corollary \ref{C:breadth=dp} and Fact \ref{F:sum-of-breadth}. As for the first part, assume that $A\oplus B$ is not strong. By Proposition \ref{P:inpattern-in-1based}, there exist p.p. formulas $\{\varphi_i(x)\}_{i<\omega}$ which define subgroups that witness this. Since, by Fact \ref{F:pp-commutes-with-sum}, $\varphi_i(A\oplus B)=\varphi_i(A)\oplus \varphi_i(B)$ for every $i<\omega$, an easy computation gives that for every $i_0<\omega$
\begin{align*}
\infty & = \left[\bigcap_{i\neq i_0} \varphi_i(A\oplus B):\bigcap_i \varphi_i(A\oplus B)\right] \\ & = \left[\bigcap_{i\neq i_0}\varphi_i(A):\bigcap_i \varphi_i(A)\right]\cdot \left[\bigcap_{i\neq i_0}\varphi_i(B):\bigcap_i \varphi_i(B)\right].
\end{align*} 
This yields that either $A$ or $B$ are not strong, contradiction.
\end{proof}

\subsection{Vapnik–Chervonenkis density} For the sake of completeness, before stating the main result of the section we briefly recall the definition of vc-density. We refer the reader to \cite{ADHMS1} for a detailed exposition. 

Let $X$ be a set and let $\mathcal S$ be a class of subsets of $X$. A subset $A$ of $X$ is \emph{shattered} by $\mathcal S$ if for any subset $A_0$ of $A$ there is a subset $C$ in $\mathcal S$ such that $A_0 = A\cap C$. 
The {\em shatter function} $\pi_{\mathcal S}:\mathbb N\rightarrow \mathbb N$ is defined as
$$
\pi_{\mathcal S}(n) = \max\{ |\{C\cap A:C\in\mathcal S\}| : |A| = n \}. 	
$$
Note that $\pi_{\mathcal S}(n)\le 2^n$ for any $n$ and we say that $\mathcal S$ has {\rm VC}-dimension at most $n$ if $\pi_{\mathcal S}(n) < 2^{n}$. A fundamental fact (independently due to Perles, Sauer, Shelah and Vapnik-Chervonenkis) yields that if $\mathcal S$ has {\rm VC}-dimension at most $d$, then $\pi_{\mathcal S}(n)$ is bounded above by a polynomial in $n$ of degree $d$, {\it i.e.} $\pi_{\mathcal S}(n) = O(n^d)$. Hence, it makes sense to define the {\em {\rm vc}-density} ${\rm vc}(\mathcal S)$ of $\mathcal S$ as the infimmum $r\ge 0$ for which we have $\pi_{\mathcal S}(n)=O(n^r)$. Thus, we have that 
$$
{\rm vc}(\mathcal S) = \limsup_{n\rightarrow \infty} \frac{\log (\pi_{\mathcal S}(n))}{\log n}.
$$ 
Model-theoretically, one can interpret these notions as follows. Let $\mathbb M$ be an enough saturated model of a first-order theory $T$ and let $\Delta$ be a finite set of formulas $\varphi(x;y)$ in the tuple of object variables $x$ and tuple of
parameter variables $y$. Denote by
$$
\mathcal S_{\Delta} = \{\varphi(\mathbb M;b) : \varphi(x;y)\in \Delta \text{ and } b\in \mathbb M^{|y|}\}
$$
the family of subsets of $\mathbb M^{|x|}$ defined by formulas $\varphi(x;y) \in \Delta$ using parameters $b$ ranging over $\mathbb M^{|y|}$. We simply write $\mathcal S_\varphi$ when $\Delta=\{\varphi(x;y)\}$. It can be shown that the shatter function of $\mathcal S_{\Delta}$ and its {\rm vc}-density do not depend on the choice of the model, see \cite[Lemma 3.2]{ADHMS1}. 

\begin{definition} The {\rm vc}-density of a theory $T$ is the function ${\rm vc}^T:\mathbb N \rightarrow \mathbb R^{\ge 0}\cup \{\infty\}$ defined as 
$$
{\rm vc}^{T} (n) =\sup \{ {\rm vc}(\mathcal S_{\varphi}):  \text{ $\varphi(x;y)$ is a formula of $T$ with } |y|=n \} .
$$
We say that a theory $T$ has {\rm vc}-density at most $d$ if ${\rm vc}^T(1)\le d$. 
\end{definition}
In fact, the {\rm vc}-density of the theory is closely related with the number of $\Delta$-types over finite sets. Recall that a complete $\Delta$-type over a set $A$ is a maximal consistent set of formulas of the form $\varphi(x;a)$ or $\neg\varphi(x;a)$ with $a\in A^{|y|}$. As usual denote by $S_{\Delta}(A)$ the space of complete $\Delta$-types over $A$. The {\em dual} of the partitioned formula $\varphi(x;y)$ is the formula $\varphi^*(y;x) = \varphi(x;y)$, where the role of the object and parameters variables are interchanged. Set $\Delta^*$ the set of dual formulas of formulas from $\Delta$. If $A$ is a subset of $\mathbb M$, then we can associate to each complete $\Delta$-type over $A$ its set of realizations in $\mathbb M$, getting a bijection between $S_{\Delta}(A)$ and $\mathcal S_{\Delta^*}$. Hence, one can check that
$$
\pi_{\mathcal S_{\Delta^*}}(n) = \max\{ |S_{\Delta}(A) | : A\subseteq M \text{ and } |A|=n\}.
$$
In addition, we also have
$$
{\rm vc}^{T} (n) =\sup \{ {\rm vc}(\mathcal S_{\varphi^*}):  \text{ $\varphi(x;y)$ is a formula of $T$ with } |x|=n \}.
$$

Regarding one-based groups, we have the following result which generalizes Proposition 4.10 and Corollary 4.13 from \cite{ADHMS}:

\begin{theorem}\label{T:DP-VC-Br}
Let $G$ be a one-based group. Then for any positive integer $m$:
$$
\dprk(G^m) = \mathrm{vc}^{\mathrm{Th}(G)}(m).
$$
Furthermore, if both are finite, then they agree with the breadth of the semilattice of commensurability classes of the $\acl^{\rm eq}(\emptyset)$-definable subgroups of $G^m$. 
\end{theorem}
\begin{proof}
By \cite[Proposition 8.2]{Farre}, we already know that $\dprk(G^m)$ is bounded above by $\mathrm{vc}^{\mathrm{Th}(G)}(m)$. To show the other inequality, suppose that $\mathrm{vc}^{\mathrm{Th}(G)}(m) \ge d$ for some positive integer $d$. Then, \cite[Proposition 4.10]{ADHMS} yields that the semilattice of commensurability classes of $\acl^{\rm eq}(\emptyset)$-definable subgroups of $G^m$ has breadth at least $d$ and so $\dprk(G^m)\ge d$ by Corollary \ref{C:breadth=dp}.
\end{proof}

Since in a stable theory the weight and hence the dp-rank is sub-additive, see for instance \cite[Proposition 5.6.5]{Buechler}, we easily obtain:

\begin{corollary}\label{C:VC-Lin}
The {\rm vc}-density of the theory of a one-based group has linear 
growth, provided that the group has finite {\rm dp}-rank. Namely, if $G$ has finite {\rm dp}-rank, then 
$$
{\rm vc}^{\mathrm{Th}(G)}(m) = m \cdot \mathrm{vc}^{\mathrm{Th}(G)}(1).
$$
\end{corollary}

\section{Abelian groups: Study by cases}\label{s:cases}

In the light of Corollary \ref{C:BasicEquation} and Szmielew's classification, to describe the $\dprk$-rank (or equivalently the {\rm vc}-density) of an abelian group we first should analyse the distinct families of infinite abelian groups appearing in a strict Szmielew group. 

We first prove some basic lemmas. We begin with the following fact on strong groups proved in \cite[Corollary 4.6]{cher-kap-sim}. We offer a proof for the sake of completeness. 

\begin{fact}\label{F:Strong}
If $A$ is a strong abelian group, then there is only a finite number of primes $p$ such that $A/pA$ is infinite. If in addition $A$ has $\dprk$-rank $k$, then there are at most $k$ many of these primes. 
\end{fact}
\begin{proof}
Let $\mathcal P$ be the set of primes $p$ such that $A/pA$ is infinite. Since $A$ is elementarily equivalent to a strict Szmielew group, we may assume by Corollary \ref{C:BasicEquation} that $A$ is indeed the group 
$$
\bigoplus_{p\in\mathcal P} \left( \bigoplus_{n>0} \mathbb{Z}(p^n)^{(\alpha_{p,n})} \oplus (\mathbb Z_{(p)})^{(\beta_p)} \right)
$$
with for each prime $p$ in $\mathcal P$ either $\beta_p=\omega$, $\alpha_{p,n}=\omega$ for some $n$ or $\alpha_{p,n}\neq 0$ for infinitely many $n$.  Then, the family $\{pA\}_{p\in \mathcal P}$ witnesses that $A$ has an inp-pattern of depth $|\Pp|$. This yields the result.
\end{proof}

\begin{lemma}\label{L:UnbddExp+Dpmin}
Let $A\oplus B$ be a strong abelian group of unbounded exponent such that every  p.p.-definable subgroup of $B$ of unbounded exponent has finite index and every p.p.-definable subgroup of $B$ of finite exponent is finite. Then 
$$
\dprk(A \oplus B)= 
\begin{cases}
\dprk(A)+ 1 & \text{if $A$ has finite exponent},\\
\dprk(A) & \text{otherwise}.
\end{cases}
$$
Moreover, in the latter case, a family of p.p.-definable subgroups witnesses an inp-pattern in $A\oplus B$ if and only if it witnesses an inp-pattern in $A$.  
\end{lemma}
\begin{proof}
Suppose first that $A$ has unbounded exponent. By Corollary \ref{C:BasicEquation} we know that $\dprk(A)\le \dprk(A\oplus B)$. To prove the other inequality, it suffices to consider an {\rm inp}-pattern of finite depth $\kappa$ in $A\oplus B$, which is witnessed by a family of p.p.-definable subgroups $\{N_i\}_{i<\kappa}$ by Remark \ref{R:Inp-pattern}. Each $N_i$ is defined by a formula $\varphi_i(x)$ of the form 
$$
\bigwedge_{j} (n_{i,j}|m_{i,j}x)\wedge (m_ix=0)
$$
with $0 \le m_{i,j} < n_{i,j}$ and $0\le m_i$. Note that the left hand side formula defines a subgroup of unbounded exponent in $A\oplus B$ by Lemma \ref{L:p.p.-non-trivial}. Thus, we have $N_i = \varphi_i(A) \oplus \varphi_i(B)$ where $\varphi_i(B)$ is finite if $m_i>0$ and it has finite index in $B$ otherwise. 

We see that the family $\{\varphi_i(A)\}_{i<\kappa}$ yields an {\rm inp}-pattern of depth $\kappa$ in $A$. Otherwise, by Proposition \ref{P:inpattern-in-1based} there exists some index $i_0<\kappa$ such that 
$$
\left[ \bigcap_{i\neq i_0} \varphi_i(A) : \bigcap_i \varphi_i(A) \right] < \infty.
$$
Necessarily, this implies, together with the assumption on $\{N_i\}_{i<\kappa}$, that 
$$
\left[ \bigcap_{i\neq i_0} \varphi_i(B) : \bigcap_i \varphi_i(B) \right] = \infty.
$$
Thus, the subgroup $\bigcap_{i\neq i_0} \varphi_i(B)$ is infinite. We see using the assumption that $\varphi_i(B)$ has unbounded exponent for $i\neq i_0$ and hence $m_i=0$ for $i\neq i_0$. Again by assumption, the subgroup $\varphi_{i_0}(B)$ must have finite exponent and so $m_{i_0} > 0$. Therefore, as $A$ has unbounded exponent, the p.p.-definable subgroup $\bigcap_{i\neq i_0} \varphi_i(A)$ has unbounded exponent by Lemma \ref{L:p.p.-non-trivial}, a contradiction since a finite exponent group cannot have finite index in a group of unbounded exponent.

Assume now that $A$ has finite exponent, say $m$. Since $B$ must have unbounded exponent, it follows immediately from the assumption that $B$ is $\dprk$-minimal and so we only need to show that $\dprk(A)+1 \leq \dprk(A\oplus B) $, by Corollary \ref{C:BasicEquation}. For this, let $\{\varphi_i(x)\}_{i<\kappa}$ be a family of p.p. formulas defining subgroups that witness an {\rm inp}-pattern of finite depth $\dprk(A)=\kappa$ in $A$. Consider the following subgroups 
$$
\varphi_i(A)\oplus (\varphi_i(B)+mB) \ \text{for $i<\kappa$} \ \text{ and } \ A\oplus B[m],
$$ 
which are p.p.-definable since $m(A\oplus B)=\{0\}\oplus mB$ and $(A\oplus B)[m] = A\oplus B[m]$. Now, using the fact that $(mB)[m]$ has infinite index in $mB$ by our initial assumptions and using Proposition \ref{P:inpattern-in-1based}, it is easy to see that the family formed by all these subgroups yields an inp-pattern of depth $\kappa+1$.
\end{proof}

\subsection{Torsion-free groups} A torsion-free abelian group $A$ is elementarily equivalent to one of the following strict Szmielew groups:
$$
\mathbb Q^{(\omega)} \ \text{ or } \ \bigoplus_{p} (\mathbb Z_{(p)})^{(\beta_p)}.
$$
Note that the former group is divisible and so $\dprk$-minimal, as well as is each summand in the second group. Indeed, their lattice of commensurability classes of p.p.-definable subgroups is a chain. Also, note that in a torsion-free abelian group, every non-trivial p.p.-definable subgroup is given by a formula of the form $\bigwedge_j (n_{j}|x)$.
Moreover, in the second group the p.p. formula $p|x$ yields an infinite index subgroup if and only if $\beta_p=\omega$. That is, the set $\mathrm{Tf}_{\ge\aleph_0}(A)$ consists of all primes $p$ such that $pA$ has infinite index in $A$.

\begin{proposition}\label{P:T-FStrong}
Let $A$ be a torsion-free abelian strict Szmielew group.  Then
$A$ is strong if and only if $\mathrm{Tf}_{\geq\aleph_0}(A)$ is finite, and then the {\rm dp}-rank is given by
$$
\dprk(A)  = \max\{1,|\mathrm{Tf}_{\ge\aleph_0}(A)|\}.
$$ 
Furthermore, an inp-pattern of maximal depth is given by the formula $x=y$ if $\beta_p<\omega$ for every $p$ and by the family of formulas $p^l|x$ for $p$ with $\beta_p=\omega$ and any integer $l\geq 1$, otherwise.
\end{proposition}
\begin{remark}
By Section \ref{ss:szmielew groups}, the equation for the {\rm dp}-rank holds for any torsion-free abelian group. 
\end{remark}
\begin{proof} 
 By the discussion above it is enough to consider the case when $A$ is 
$$
\bigoplus_{p} (\mathbb Z_{(p)})^{(\beta_p)}.
$$ 
If $A$ is strong then $\mathrm{Tf}_{\geq\aleph_0}(A)$ is finite by Fact \ref{F:Strong}. For the other direction, put $A=B\oplus C$, where $B$ is the direct sum of all $(\mathbb{Z}_{(p)})^{(\beta_p)}$ for which $\beta_p=\omega$ and $C$ is direct sum of all of those groups for which $\beta_p$ is finite. Since $B$ is a finite direct sum of $\dprk$-minimal groups (indeed, their p.p.-definable subgroups form a chain), by using Corollary \ref{C:BasicEquation} it remains to show that $C$ is strong. The abelian group $C$ has no non-trivial p.p.-definable subgroups of infinite index, hence it is $\dprk$-minimal and so  strong.

We assume $A$ is strong.
By Lemma \ref{L:UnbddExp+Dpmin}, we may further assume that $\beta_p=\omega$ for every $p$. Hence, we immediately obtain that $\dprk(A)=|\mathrm{Tf}_{\ge\aleph_0}(A)|$ by Corollary \ref{C:BasicEquation} and again Fact \ref{F:Strong}. 

Finally, the last part of the statement follows from Fact \ref{F:Strong}, since the only non-trivial p.p-definable proper  subgroups of $A$ are given by formulas of the form $m|x$ for $0<m$.
\end{proof}

\begin{corollary}\label{C:TF}
If $A$ is a strong torsion-free abelian group, then it has finite $\dprk$-rank which is bounded above by the weight of its generic type.  
\end{corollary}
\begin{proof}
The equation from the previous proposition yields that a strong torsion-free abelian group $A$ has finite $\dprk$-rank. In particular, some (any) generic type of $A$ has finite weight, say $k$, and therefore, there are only $k$ many primes $p$ such that $A/pA$ is infinite by \cite[Proposition 2.1]{KruPil}. Consequently, the $\dprk$-rank of $A$ is at most $k$ by the previous result.
\end{proof}
\subsection{Divisible torsion groups} A divisible torsion abelian group $A$ is elementarily equivalent to a group of the form:
$$
\bigoplus_{p} \mathbb Z(p^\infty)^{(\gamma_p)}.
$$
It is easy to see that every p.p.-definable subgroup of a divisible torsion group is given by a formula of the form $mx=0$ for some integer $m\geq 0$.
Using the description of p.p-definable groups, it is easy to see that the group $\mathbb Z(p^\infty)$ is $\dprk$-minimal since each p.p.-definable proper subgroup is finite. Furthermore, recall that a prime $p$ belongs to $\mathrm{D}_{\ge\aleph_0}(A)$ if and only if $\gamma_p = \omega$. 

\begin{proposition}\label{P:Divisible+Torsion}
If $A$ is a divisible torsion abelian strict Szmielew group, then it is $\omega$-stable, hence strong, and 
$$
\dprk(A) = \max\{1,|\mathrm{D}_{\ge\aleph_0}(A)|\}.
$$
Furthermore, an inp-pattern of depth $n$ is witnessed by the family of p.p.-definable subgroups 
$$
\left\{\bigoplus_{q\neq p} \mathbb Z(q^\infty)^{(\omega)}[q^r]\right\}_{p\in \mathcal{P}}
$$
for any fixed positive integer $r$ and $\mathcal{P}$ a subset of $\mathrm{D}_{\geq\aleph_0}(A)$ of size $n$.
\end{proposition}
\begin{remark}
By Section \ref{ss:szmielew groups}, this is true for any divisible torsion abelian group.
\end{remark}
\begin{proof}
Let $A$ be the group
$$
\bigoplus_{p} \mathbb Z(p^\infty)^{(\gamma_p)}.
$$ 
Using the description of p.p. formulas, any p.p.-definable proper subgroup of $A$ must be given by a formula equivalent to $mx=0$ with $m>0$. Consequently, it can be easily seen that $A$ satisfies the minimal chain condition on p.p.-definable subgroups and so the theory of $A$ is $\omega$-stable. Furthermore, this also yields that $\dprk(A)=1$ whenever $\gamma_p<\omega$ for every prime $p$. 

Assume now that there are some primes $p$ with $\gamma_p=\omega$, {\it i.e.} $\mathrm{D}_{\ge\aleph_0}(A)$ is non-empty. We show that $\dprk(A)=|\mathrm{D}_{\ge\aleph_0}(A)|$. We emphasize here that $\dprk(A)$ might be $\aleph_0$ even though $A$ is strong. Fix an integer $r$ and consider a finite subset $\mathcal P$ of $\mathrm{D}_{\ge\aleph_0}(A)$. Note that, since $\mathbb{Z}(q^\infty)^{(\omega)}[q^r]$ is infinite, the family of p.p.-definable subgroups 
$$
\left\{\bigoplus_{q\in \mathcal P\setminus\{q'\}} \mathbb Z(q^\infty)^{(\omega)}[q^r] \right\}_{q'\in \mathcal P} 
,$$ defined by the formulas $\left(\prod_{q\in \mathcal{P} \setminus \{q'\}} q^r \right)x=0$,
witnesses an inp-pattern of depth $|\mathcal P|$ in $A$. Hence, we get that $\dprk(A) \ge |\mathrm{D}_{\ge\aleph_0}(A)|$. For the other inequality, assume as we may that $\mathrm{D}_{\ge \aleph_0}(A)$ is finite. Note that each p.p.-definable subgroup of 
$$
\bigoplus_{p ; \ \gamma_p<\omega} \mathbb Z(p^\infty)^{(\gamma_p)}
$$ 
is the whole group if it has unbounded exponent and is finite otherwise. Hence, we then have that $\dprk(A)\le |\mathrm{D}_{\ge\aleph_0}(A)|$ by Lemma \ref{L:UnbddExp+Dpmin} and Corollary \ref{C:BasicEquation}, since each $\mathbb Z(p^\infty)^{(\omega)}$ is $\dprk$-minimal. 
\end{proof}

\begin{remark}
It might be worthwhile to note that $\bigoplus_{p \text{ prime }} \mathbb{Z}{(p^\infty)}^{(\omega)}$ is an example of an $\omega$-stable, not of finite {\rm dp}-rank abelian group in the pure language of groups.
\end{remark}

\subsection{Finite exponent groups} Abelian groups of finite exponent are elementarily equivalent to a group of the form
$$
\bigoplus_p \bigoplus_{n>0} \mathbb Z(p^n)^{(\alpha_{p,n})}
$$
with $\alpha_{p,n}\neq 0$ only for finitely many pairs $(p,n)$. Their breadth has been characterized in \cite[Theorem 5.1]{ADHMS} in terms of the Ulm invariants. Recall that the {\em Ulm invariants} of an abelian group $A$ are defined for each prime $p$ and natural number $n$ as 
$$
\mathrm{U}(p,n;A) = \left| (p^n A)[p] / (p^{n+1}A)[p] \right|.
$$ 
For each prime $p$, we set
$$
\mathrm{U}(p;A) = \{ n\ge 0 : \mathrm{U}(p,n;A) >1\}
$$
and note that the set $\mathrm{U}(p;A)$ is finite if and only if $A$ has finite $p$-length, and that $A$ has finite exponent if and only if $A$ is torsion and the set $\mathrm{U}(p;A)$ is finite with $\mathrm{U}(p;A)= \emptyset$ for all but finitely many $p$. 

Recall that, given a finite nonempty set $I$, we
defined $d(I)$ as the size of a maximal subset $I_0$ of $I$ such that $j-i\ge 2$ for any two integers $i<j$ from $I_0$\footnote{In \cite{ADHMS}, the definition of $d(I)$ differs from ours, but both formulations are easily seen to be equivalent.}.

\begin{fact}\cite[Theorem 5.1]{ADHMS}\label{F:FiniteExpBreadth}
Suppose that $A$ has finite exponent, and set 
$$
d  = \sum_p  d(\mathrm{U}(p;A)).
$$
Then, the lattice of the p.p.-definable subgroups of $A$ has breadth $d$.
\end{fact}

For each prime $p$ recall that
$$
\mathrm{U}_{\ge \aleph_0}(p;A) = \{ n\ge 0 : \mathrm{U}(p,n;A) \text{ is infinite} \}.
$$

\begin{proposition}\label{P:BoundedExponent}
If $A$ is an infinite abelian group of finite exponent then its $\dprk$-rank is finite and is equal to 
$$
\sum_p d \big( \mathrm{U}_{\ge \aleph_0}(p;A) \big).
$$
\end{proposition}
\begin{proof}
Assume that $A$ has finite exponent, and since there is not harm in assuming that 
\[A=\bigoplus_p \bigoplus_{n>0} \mathbb Z(p^n)^{(\alpha_{p,n})},\] we can write $A$ as $B\oplus C$, where $C$ is a finite group and $B$ is such that $\mathrm{U}(p;B) = \mathrm{U}_{\ge \aleph_0}(p;A)$ for every prime $p$. 

Since $B^{(\omega)}\equiv B$, by \cite[Section 4.6]{ADHMS} the breadth of the lattice of p.p.-definable subgroups of $B$ is equal to that of the lattice of commensurability classes. It then follows by the previous fact that the lattice of commensurability classes of p.p.-definable subgroups of $B$ has breadth $\sum_p d(\mathrm{U}_{\ge \aleph_0}(p;A))$ and so we obtain the result by Corollary \ref{C:breadth=dp}, since $C$ is finite.
\end{proof}

\subsection{Torsion of unbounded length} A torsion abelian group whose primary $p$-components $A_p$ have unbounded $p$-length is elementary equivalent to a group of the form 
$$
\bigoplus_p \bigoplus_{n>0} \mathbb Z(p^n)^{(\alpha_{p,n})}
$$
with $\alpha_{p,n}=0$ only for finitely many $n$ for each prime $p$. Of course, the number of coefficients $\alpha_{p,n}$ that are non-zero or infinite determines the structure of the semilattice of commensurable classes of the p.p.-definable subgroups. We first analyse when these groups are strong.

We start with this easy observation, which follows easily by Remark \ref{R:p^n|p^mx}.
\begin{fact}\label{F:pp-formula}
Let $\mathcal X$ be a subset of $\omega$ and let $A$ be the group 
\[
\bigoplus_{n\in\mathcal X} \mathbb{Z}(p^n)^{(\alpha_{p,n})}.
\] 
Then, the formula $p^r|p^sx$ with $0\leq s<r$ defines the subgroup
\[\bigoplus_{n\le s} \mathbb{Z}(p^n)^{(\alpha_{p,n})} \oplus  \bigoplus_{s<n\le r} \mathbb Z(p^n)^{(\alpha_{p,n})}[p^s] \oplus \bigoplus_{r<n} p^{r-s}\mathbb{Z}(p^n)^{(\alpha_{p,n})} .\]
\end{fact}

\begin{lemma}
An abelian group of the form
$$
\bigoplus_{n\in\mathcal X} \mathbb  Z(p^n)^{(\omega)}
$$ 
is strong if and only if the set $\mathcal X$ is finite if and only if it has finite dp-rank.
\end{lemma}
\begin{proof}
Let $A$ be an abelian group of this form. Since the lattice of subgroups of $\mathbb Z(p^n)$ is linearly ordered, so is the lattice of p.p.-definable subgroups of $\mathbb Z(p^n)^{(\omega)}$. Consequently, we see that $\mathbb Z(p^n)^{(\omega)}$ is dp-minimal and so we have by Corollary \ref{C:BasicEquation} that $A$ has finite dp-rank whenever $\mathcal X$ is finite. As an abelian group of finite dp-rank is strong, it remains to show that the set $\mathcal X$ is finite whenever $A$ is strong. To do so, we assume that $\mathcal X$ is infinite and we construct an inp-pattern of depth $\omega$.

Suppose that $\mathcal X$ is infinite and let $(n_i)_{i<\omega}$ be a strictly increasing sequence of elements of $\mathcal X$ with the property that $2n_i<n_{i+1}$. Now, consider for each $i<\omega$ the formula $p^{2n_i}|p^{n_i}x$, which we denote by $\varphi_i(x)$. We claim that the set $\{\varphi_i(x)\}_{i<\omega}$ is an inp-pattern of depth $\omega$ for $A$. To see this, set $C_n$ to denote the group $\mathbb Z(p^n)^{(\omega)}$ and note that it suffices to show that for any $k$ we have that
\[
\bigcap_{i\neq k} \varphi_i(C_{n_{k+1}}) \neq \bigcap_{i} \varphi_i(C_{n_{k+1}}).
\] 
We see using Fact \ref{F:pp-formula} that $\varphi_i(C_{n_{k+1}}) = p^{n_i}C_{n_{k+1}}$ for $i< k+1$, since $2n_i<n_{i+1}$, and also that $\varphi_i(C_{n_{k+1}})=C_{n_{k+1}}$ for $i\ge k+1$. It then follows for each $k\ge 1$ that
\[
\bigcap_{i\neq k} \varphi_i(C_{n_{k+1}}) =  p^{n_{k-1}}C_{n_{k+1}} \gneq p^{n_k}C_{n_{k+1}} = \bigcap_{i} \varphi_i(C_{n_{k+1}}),
\] 
and for $k=0$ that
\[
\bigcap_{i>0} \varphi_i(C_{n_{1}}) =  C_{n_{1}} \gneq p^{n_0}C_{n_{1}} = \bigcap_{i} \varphi_i(C_{n_{1}}).
\] 
This yields that $A$ is not strong.
\end{proof}

\begin{lemma}\label{L:UnbddExp-Strong}
An abelian group of the form
$$
\bigoplus_{p}\bigoplus_{n>0} \mathbb  Z(p^n)^{(\alpha_{p,n})}
$$ 
is strong if and only if it has finite dp-rank if and only if there is only a finite number of pairs $(p,n)$ such that $\alpha_{p,n}=\omega$ and finitely many primes $p$ such that $\alpha_{p,n}\neq 0$ for infinitely many $n$.
\end{lemma}
\begin{proof}
Let $A$ be such a group and let $\Pp$ be the set of primes such that $\alpha_{p,n}\neq 0$ for infinitely many $n$ or $\alpha_{p,n}=\omega$ for some $n$. Note that for each prime $p$ in $\Pp$, the subgroup $pA$ has infinite index and the set $\Pp$ must be finite by Fact \ref{F:Strong} when $A$ is strong.  Furthermore, the previous lemma yields that for each prime $p$ in $\mathcal P$ there is only finitely many natural numbers $n$ such that $\alpha_{p,n}=\omega$. 

Assume now that there is only a finite number of pairs $(p,n)$ such that $\alpha_{p,n}=\omega$ and finitely many primes $p$ such that $\alpha_{p,n}\neq 0$ for infinitely many $n$. Thus, we then have that $\Pp$ is finite. 
Write $A$ as 
\[
\left( 
\bigoplus_{p\in \Pp}\bigoplus_{n>0} \mathbb  Z(p^n)^{(\alpha_{p,n})}
\right)\oplus \left( 
\bigoplus_{p\notin \Pp}\bigoplus_{n>0} \mathbb  Z(p^n)^{(\alpha_{p,n})}
\right).
\]
Note that the right summand is a non-singular abelian group in the sense of \cite{ADHMS}, and so it is dp-minimal by \cite[Proposition 5.27]{ADHMS}. In particular, it is strong. Together with the fact that $\Pp$ is finite, by Corollary \ref{C:BasicEquation} it is enough to show that for each prime $p\in\mathcal P$ 
the primary $p$-component $A_p$ of $A$ has finite dp-rank. Hence, without loss of generality, assume $A=A_p$ and $p\in\mathcal P$. Furthermore, since the group    
\[
\bigoplus_{\substack{n>0 \\ \alpha_{p,n}<\omega}} \mathbb  Z(p^n)^{(\alpha_{p,n})}
\]
is dp-minimal by 
\cite[Proposition 5.27]{ADHMS}, we may further assume that 
\[
A = \bigoplus_{n\in\mathcal X} \mathbb Z(p^n)^{(\omega)},
\]
where $\mathcal X$ is a subset of $\omega$. Note that $\mathcal X$ is finite by assumption and so $A$ has finite dp-rank by the previous lemma, as desired.
\end{proof}

\begin{lemma}\label{L:unbounded-fin-exponents}
Let $A$ be an abelian group of the form
$$
\bigoplus_{p\in\Pp}\bigoplus_{n>0} \mathbb  Z(p^n)^{(\alpha_{p,n})}
$$ 
with $\alpha_{p,n}\neq 0$ for infinitely many $n$ for every prime $p$ and $\alpha_{p,n}$ finite for every pair $(p,n)$, and assume that $\Pp$ is a finite set. Then, the following holds:
\begin{enumerate}
\item Given positive integers $r,s$, the subgroup $A[r]$ is contained up to finite index in $sA$, {\it i.e.} the subgroup $A[r]\cap sA$ has finite index in $A[r]$.
\item The $\dprk$-rank of $A$ equals $|\Pp|$. Moreover, for each positive integer $r$ the families $\{p^rA\}_{p\in\Pp}$ and $\{\bigoplus_{q\neq p} A[q^r]\}_{p\in\Pp}$ yield inp-patterns of depth $|\Pp|$.
\end{enumerate}
\end{lemma}
\begin{proof}
Since $\Pp$ is finite, to show $(1)$ it suffices to show that $\mathbb Z(p^n)[r]$ is contained in $s\mathbb Z(p^n)$ for large enough $n$. Write $r = p^m k$ with $p$ and $k$ coprime, and note then that
$$
\mathbb Z(p^n)[r] = \mathbb Z(p^n)[p^m] = p^{n-m}\mathbb Z(p^n)
$$ 
is contained in $p^t \mathbb Z(p^n) = s\mathbb Z(p^n)$, for $s=p^tk'$ with $p$ and $k'$ coprime, for $n\ge m+t$, as desired.

For $(2)$, since $A$ is a direct sum of $|\Pp|$ many $\dprk$-minimal groups by \cite[Proposition 5.27]{ADHMS}, we have that $\dprk(A)\le |\Pp|$ by Corollary \ref{C:BasicEquation}. To obtain the equality, note that for a given $r>0$, every subgroup $p^r A$ has infinite index in $A$ and so we immediately get that the groups $p^rA$ for $p\in \Pp$ form an inp-pattern. 

On the other hand, to get an inp-pattern using finite exponent subgroups, note that for each prime $p_0$ in $\Pp$ we have that
$$
\left[ \bigcap_{p\in\Pp\setminus \{p_0\}} \bigoplus_{q\neq p} A[q^r] : \bigcap_{p\in\Pp} \bigoplus_{q\neq p} A[q^r]\right] =
\left| A[p_0^r] \right|,
$$ 
which is infinite. Thus, the family of groups $\bigoplus_{q\neq p} A[q^r]$ for $p\in\Pp$ also forms an inp-pattern of depth $|\Pp|$. 
\end{proof}

Recall that the set of primes $p$ such that an abelian group $A$ has unbounded $p$-length is denoted by $\mathrm{U}_{\ge\aleph_0}(A)$. Moreover, when $A$ is a strict Szmielew group, then this set is precisely the collection of primes $p$ with $\alpha_{p,n}\neq 0$ for infinitely many $n$. Finally, recall that we denoted by $\mathrm{U}_{\ge \aleph_0}(p;A)$ the set of $n$ such that $\alpha_{p,n}$ is infinite.

\begin{proposition}\label{P:Torsion+UnbddLength}
If $A$ is a strong abelian strict Szmielew group of unbounded exponent of the form
$$
\bigoplus_p \bigoplus_{n>0} \mathbb Z(p^n)^{(\alpha_{p,n})},
$$
then it has finite dp-rank and we have
$$
\dprk(A) = \max \left\{ 1 ,|\mathrm{U}_{\ge\aleph_0}(A)| \right\} + \sum_{p} \big(d\left(\mathrm{U}_{\ge\aleph_0} (p;A)\right)\big) .
$$
Moreover, if $\mathrm{U}_{\ge\aleph_0}(A)$ is non-empty, then an inp-pattern of maximal depth can be formed by p.p.-definable subgroups all of finite exponent or all of unbounded exponent.
\end{proposition}
\begin{proof}
By Lemma \ref{L:UnbddExp-Strong}, there is only a finite number of pairs $(p,n)$ such that $\alpha_{p,n}=\omega$ and only finitely many primes $p$ such that $\alpha_{p,n}\neq 0$ for infinitely many $n$. In particular, the sets $\mathrm{U}_{\ge\aleph_0}(A)$ and $\mathrm{U}_{\ge\aleph_0}(p;A)$ are finite. Hence, for each prime $p$ we can set $n_p$ to be the largest integer $n$ such that $\alpha_{p,n}=\omega$ and set $n_p=0$ if it does not exist. 

Now, we can write $A$ as $B\oplus C$ where
$$
B= \bigoplus_{p} \bigoplus_{n \le n_p} \mathbb Z(p^n)^{(\alpha_{p,n})} \ \text{ and } \ C= \bigoplus_{p} \bigoplus_{n>n_p} \mathbb Z(p^n)^{(\alpha_{p,n})}.
$$
Note that $B$ has finite exponent and we then have that $\dprk(B) =\sum_p d(\mathrm{U}_{\ge\aleph_0}(p;B))$  by Proposition \ref{P:BoundedExponent}. Moreover, the infinite group $C$ can be written as $C_1\oplus C_2$, where $C_1$ is the direct sum of all infinite primary components, and hence a finite sum, whereas $C_2$ is only the direct sum of the finite ones. Therefore, any infinite p.p.-definable subgroup of $C_2$ of unbounded exponent must have finite index and every p.p.-definable subgroup of bounded exponent is finite.
 
Assume that $C_1$ is trivial. Since $B$ has finite exponent, applying Lemma \ref{L:UnbddExp+Dpmin} we obtain that 
$$
\dprk(A) =  1 + \dprk(B) = 1+ \sum_p d \big(\mathrm{U}_{\ge\aleph_0}(p;B)\big)
.$$
This agrees with the equation in the statement because the triviality of $C_1$ implies that $\mathrm{U}_{\geq \aleph_0}(A)=\emptyset$ and $\mathrm{U}_{\ge\aleph_0}(p;B)=\mathrm{U}_{\ge\aleph_0}(p;A)$ by the definition of $B$.

Now, assume that $C_1$ is non-trivial, {\it i.e.} $\mathrm{U}_{\ge\aleph_0}(A)$ is non-empty. Again by Lemma \ref{L:UnbddExp+Dpmin} we get that 
$$
\dprk(A) = \dprk(B\oplus C) = \dprk(B\oplus C_1)
$$ 
and moreover note that $\dprk(C_1)=|\mathrm{U}_{\ge\aleph_0}(A)|$ by Lemma \ref{L:unbounded-fin-exponents}(2), which is finite by our first observation. Furthermore, we also have by Lemma \ref{L:UnbddExp+Dpmin} that $\dprk(C)=\dprk(C_1)$ and that any set of p.p.-definable subgroups that witness an inp-pattern in $C$ also witness an inp-pattern (of the same depth) in $C_1$, and {\it viceversa}. Therefore, we need to show that $\dprk(A) = \dprk(B)+\dprk(C)$. The inequality $\dprk(A)\leq \dprk(B)+\dprk(C)$ follows from Corollary \ref{C:BasicEquation}. To show the other inequality, we will show that we may lift inp-patterns of $B$ and $C$ to an inp-pattern of $A$. By the above discussion, we may assume that $C=C_1$.

Let $r$ be the exponent of $B$, {\it i.e.} the product of all $n_p$ for $p$ with $\mathrm{U}_{\ge\aleph_0}(p;A)\neq \emptyset$. It then follows that
$$
rA = \{0\}\oplus rC \text{ \ and \ } A[r] = B \oplus C[r].
$$ 
Recall that $B$ has finite {\rm dp}-rank. Let $\{\varphi_i(x)\}_{i<k_1}$ be a family of p.p. formulas witnessing an inp-pattern of maximal depth in $B$, and suppose that each $\varphi_i(B)$ is a group. Moreover, since $B$ has exponent $r$, there is no harm in assuming that the formula $\varphi_i(x)$ implies $rx=0$. 
Now, set $k_2=|\mathrm{U}_{\ge\aleph_0}(A)|$ and consider a family $\{\psi_i(x)\}_{i<k_2}$ of p.p. formulas witnessing an inp-pattern in $C$. By Lemma \ref{L:unbounded-fin-exponents}(2), we may distinguish two cases.

\noindent {\bf Case 1}. Choose each $\psi_i(x)$ in a way that $\psi_i(C)$ is subgroup of $C$ of the form $p^rC$ for $p$ in $\mathrm{U}_{\ge\aleph_0}(A)$. We then have that $\psi_i(C)$ is subgroup of $rC$. Then, consider the following p.p.-definable subgroups of $A = B\oplus C$:
$$
\varphi_i(B) \oplus \left( \varphi_i(C) + rC \right) \text{ \ for $i<k_1$ \ and \ } B \oplus \left(\psi_i(C) + C[r] \right) \text{ \ for $i<k_2$}.
$$
Note that all of them have unbounded exponent. Since the $\dprk$-rank of an abelian group is the same as the breadth of the semilattice of commensurable classes of p.p.-definable groups, we can replace the groups above by commensurable ones, possibly not p.p.-definable. Now, as $C[r]$ is contained up to finite index in $rC$, we have that $rC$ has finite index in $C[r]+rC$. Since in addition $\varphi_i(C)$ is contained in $C[r]$,
we get that $\varphi_i(C)+rC$ has finite index in $C[r]+rC$ and so we can replace the groups $\varphi_i(B)\oplus (\varphi_i(C)+rC)$ by $\varphi_i(B)\oplus (C[r]+ rC)$ for $i<k_1$. Then, since $\psi_i(C)$ is contained in $rC$, it is easy to see that the subgroups
$$
\varphi_i(B) \oplus \left(C[r]+rC \right)  \text{ \ for $i<k_1$ \  and \ } B \oplus \left (C[r]+\psi_i(C) \right) \text{ \ for $i<k_2$}.
$$
yield that the $\dprk$-rank of $A$ is at least $k_1+k_2$, as desired.

\noindent {\bf Case 2}. Suppose now that each $\psi_i(x)$ defines a subgroup $C[m_i]$ with $m_i>r$. Thus, we may take $m_i$ in a way that $\psi_i(B) = B$. Let $m$ be the product of all $m_i$ and note that $(mA)[m] = \{0\}\oplus (mC)[m]$. Hence, the following subgroups of finite exponent are p.p.-definable:
$$
\varphi_i(B) \oplus \left(\varphi_i(C) + (mC)[m] \right)  \text{ \ for $i<k_1$ \  and \ } B \oplus \psi_i(C) \text{ \ for $i<k_2$}.
$$
Note that $\psi_i(C)=C[m_i]$ is contained up to finite index in $(mC)[m]=mC\cap C[m]$. Indeed, $C[m_i]\subseteq C[m]$ by definition and $C[m_i]$ is contained in $mC$ up to finite index by Lemma \ref{L:unbounded-fin-exponents}(1). 
Now it is easy to see that the family formed by all these groups yields an inp-pattern of depth $k_1+k_2$. This finishes the proof.
\end{proof}

\section{A characterisation of strong abelian groups}\label{s:char-strong}
After the previous section, we can easily characterise all abelian groups whose theory is strong, as well as of finite $\dprk$-rank.

\begin{theorem}\label{T:Strong}
An abelian group $A$ is strong if and only if there is only a finite number of primes $p$ such that $A/pA$ is infinite and for all primes $p$ the group $p^n A[p]/p^{n+1}A[p]$ is infinite only for finitely many $n$.

Furthermore, strong abelian groups are precisely those which are elementarily equivalent to a group of the form 
$$
\bigoplus_p \left( \bigoplus_{n>0} \mathbb{Z}(p^n)^{(\alpha_{p,n})}\oplus (\mathbb{Z}_{(p)})^{(\beta_p)}\oplus \mathbb{Z}(p^\infty)^{(\gamma_p)}\right) \oplus \mathbb{Q}^{(\delta)}
$$
with $\beta_p = \omega $ only for finitely many primes $p$ and there are only a finite number of pairs $(p,n)$ such that $\alpha_{p,n}=\omega$ and only finitely many primes $p$ with $\alpha_{p,n}\neq 0$ for infinitely many $n$. 
\end{theorem}
\begin{proof}
As being strong is a property of the theory, we may assume that $A$ is a strict Szmielew group
$$
\bigoplus_p \left( \bigoplus_{n>0} \mathbb{Z}(p^n)^{(\alpha_{p,n})}\oplus (\mathbb{Z}_{(p)})^{(\beta_p)}\oplus \mathbb{Z}(p^\infty)^{(\gamma_p)}\right) \oplus \mathbb{Q}^{(\delta)}.
$$
Thus, to describe when the group $A$ is strong it suffices by Corollary \ref{C:BasicEquation} to characterise when the following subgroups 
$$
\bigoplus_p \mathbb{Z}(p^\infty)^{(\gamma_p)}, \ \bigoplus_p (\mathbb{Z}_{(p)})^{(\beta_p)} \oplus \mathbb Q^{(\delta)},  \ \text{ and } \  \bigoplus_p \left( \bigoplus_{n>0} \mathbb{Z}(p^n)^{(\alpha_{p,n})}\right),
$$
are strong. Therefore, applying Proposition \ref{P:Divisible+Torsion} and \ref{P:T-FStrong} as well as Lemma \ref{L:UnbddExp-Strong} we obtain the desired characterization from the second part of the statement. For the first part, it is easy to see that any abelian group $A$ for which there are only finitely many primes $p$ with $A/pA$ infinite and that for all primes $p$ the set $\mathrm{U}_{\ge \aleph_0}(p;A)$ is finite must be elementarily equivalent to a strict Szmielew group like in the statement; the details are left to the reader. This yields the result.
\end{proof}

Concerning abelian groups of finite $\dprk$-rank, its characterisation was already obtained in \cite[Theorem 5.23]{ADHMS} modulo the equivalence between finite dp-rank and finite ${\rm vc}$-density. In fact, using the results from the previous section together with Corollary \ref{C:BasicEquation}, one can easily show the statement.

\begin{fact}\label{F:CharacFinDprank}
An abelian group has finite $\dprk$-rank if and only if there is only a finite number of primes $p$ for $A$ such that either $A/pA$ or $A[p]$ are infinite, and for all primes $p$ the group $p^n A[p]/p^{n+1}A[p]$ is infinite only for finitely many $n$.

Furthermore, abelian groups of finite $\dprk$-rank are those which are elementarily equivalent to a group of the form 
$$
\bigoplus_p \left( \bigoplus_{n>0} \mathbb{Z}(p^n)^{(\alpha_{p,n})}\oplus (\mathbb{Z}_{(p)})^{(\beta_p)}\oplus \mathbb{Z}(p^\infty)^{(\gamma_p)}\right) \oplus \mathbb{Q}^{(\delta)}
$$
with $\beta_p = \omega $ or $\gamma_p=\omega$ only for finitely many primes $p$, for each prime $p$ there are only finitely many $n$ with $\alpha_{p,n}=\omega$ and there is only a finite number of primes $p$ such that either $\alpha_{p,n}\neq 0$ for infinitely many $n$ or $\alpha_{p,n}=\omega$ for some $n$. 
\end{fact}

Let us remark that the only obstacle for a strong abelian group to have finite dp-rank is the existence of infinitely many primes with infinite $p$-torsion. More precisely, we get:

\begin{corollary}
A strong abelian group $A$ has finite dp-rank if and only if there is only a finite number of primes $p$ with $A[p]$ infinite.
\end{corollary}
\section{An equation to compute the dp-rank}\label{s:main}

Here we prove our main result which we deduce from the following one. 
%

\begin{proposition}
Let $A$ be an infinite strict Szmielew group of finite $\dprk$-rank. Define $\mathcal P_1$ to be the set of primes $p$ such that $\alpha_{p,n}\neq 0$ for infinitely $n$ and $\mathcal P_2$ to be the set of primes $p$ such that $\alpha_{p,n} =0$ for all but finitely many with at least one $\alpha_{p,n}=\omega$.

One the following holds:
\begin{enumerate}
\item $A$ is torsion-free and
$$
\dprk(A)= \max \left\{ 1,\dprk \left( \bigoplus_{\substack{ p \\ \beta_p=\omega}} (\mathbb{Z}_{(p)})^{(\beta_p)} \right) \right\}.
$$
\item $A$ is of finite exponent and 
$$
\dprk(A)= \dprk \left( \bigoplus_{p\in \Pp_2} \bigoplus_{n>0}  \mathbb{Z}(p^n)^{(\alpha_{p,n})} \right).
$$ 

\item $A$ has unbounded exponent, torsion but has finite $p$-length for every prime $p$ and

\begin{align*}
\dprk(A) = & \  \dprk \left( \bigoplus_{p\in \Pp_2} \bigoplus_{n>0}  \mathbb{Z}(p^n)^{(\alpha_{p,n})} \right) \\ & + \max \left\{ 1, \dprk \left( \bigoplus_{\substack{ p \\ \beta_p=\omega}} (\mathbb{Z}_{(p)})^{(\beta_p)} \right)  , \dprk \left( \bigoplus_{\substack{ p \\  \gamma_p=\omega}} \mathbb{Z}(p^\infty)^{(\gamma_p)} \right) \right\}.
\end{align*}
\item $A$ has unbounded $p$-length for every prime $p$ and
\begin{align*}
\dprk(A) = & \  \dprk \left( \bigoplus_{p\in \Pp_1\cup\Pp_2} \bigoplus_{n>0}  \mathbb{Z}(p^n)^{(\alpha_{p,n})} \right) \\ & + \max \left\{ \dprk \left( \bigoplus_{\substack{ p \\ \beta_p=\omega}} (\mathbb{Z}_{(p)})^{(\beta_p)} \right)  , \dprk \left( \bigoplus_{\substack{ p \\  \gamma_p=\omega}} \mathbb{Z}(p^\infty)^{(\gamma_p)} \right) \right\}.
\end{align*}
\end{enumerate}

\end{proposition}
\begin{proof}
Let $A$ be the strict Szmielew group
$$
\bigoplus_p \left( \bigoplus_{n>0} \mathbb{Z}(p^n)^{(\alpha_{p,n})}\oplus (\mathbb{Z}_{(p)})^{(\beta_p)}\oplus \mathbb{Z}(p^\infty)^{(\gamma_p)}\right) \oplus \mathbb{Q}^{(\delta)}.
$$ 
Since $A$ has finite $\dprk$-rank, we have by Fact \ref{F:CharacFinDprank} that  $\beta_p < \omega$ and $\gamma_p<\omega$ for all but finitely many $p$ and there is only a finite number of primes $p$ such that $\alpha_{p,n}\neq 0$ for infinitely many $n$ and $\alpha_{p,n}=\omega$ only for finitely many pairs $(p,n)$. 

We compute the $\dprk$-rank of $A$. To do so, since the $\dprk$-rank of a torsion-free group has already been computed in Proposition \ref{P:T-FStrong}, we may assume that $A$ is not torsion-free. Similarly, the finite exponent case was shown in Proposition \ref{P:BoundedExponent}. Thus, we also assume that $A$ has unbounded exponent. Furthermore, in case that $\delta=\omega$, the summand preceding $\oplus \, \mathbb Q^{(\omega)}$ must have finite exponent and so $\beta_p=\gamma_p=0$ for every prime $p$ and $\alpha_{p,n}=0$ for all but finitely many pairs $(p,n)$. If this subgroup of $A$ has exponent $m$, it is p.p.-definable by the formula $mx=0$. In addition, we then have that $mA = \mathbb Q^{(\omega)}$, which is $\dprk$-minimal, and so by Corollary \ref{C:BasicEquation} and Proposition \ref{P:BoundedExponent} we obtain 
$$
\dprk(A) = \dprk \left( \bigoplus_p \left( \bigoplus_{n>0} \mathbb{Z}(p^n)^{(\alpha_{p,n})} \right) \right) + \dprk \left( \mathbb Q^{(\omega)} \right) .
$$
At this point, we have already obtained the equations for $(1)$ and $(2)$, as well as a particular case of $(3)$. Hence, for the rest, we assume that $\delta=0$ and show the remaining cases. To ease notation let
\begin{itemize}
\item $\mathcal P$ to be the set of primes $p$ such that $\alpha_{p,n}\neq 0$ for some $n$,
\item $\mathcal X$ the set of primes $p$ such that  $\beta_p\neq 0$, and finally
\item $\mathcal Y$ denote the set of primes $p$ such that $\gamma_p\neq 0$.
\end{itemize} 
As $A$ is an infinite strict Szmielew group we have that some of these sets of primes are non-empty. 
We then have
$$
A=  \bigoplus_{p\in \mathcal P} \bigoplus_{n>0} \mathbb{Z}(p^n)^{(\alpha_{p,n})} \oplus \bigoplus_{p\in \mathcal X}  (\mathbb{Z}_{(p)})^{(\beta_p)} \oplus \bigoplus_{p\in \mathcal Y} \mathbb{Z}(p^\infty)^{(\gamma_p)}.
$$ 
To ease notation, set $A_{\mathcal P}$, $A_{\mathcal X}$ and  $A_{\mathcal Y}$ to be
$$
A _{\mathcal P}=  \bigoplus_{p\in \mathcal P} \bigoplus_{n>0} \mathbb{Z}(p^n)^{(\alpha_{p,n})}, \ A_{\mathcal X} = \bigoplus_{p\in \mathcal X}  (\mathbb{Z}_{(p)})^{(\beta_p)} \ \text{ and } \ A_{\mathcal Y} = \bigoplus_{p\in \mathcal Y} \mathbb{Z}(p^\infty)^{(\gamma_p)}.
$$
Now we claim the following:
\begin{claim}
$\dprk(A) = \max\{\dprk(A_{\Pp} \oplus A_\X) , \dprk(A_{\Pp} \oplus A_\Y)\}$.
\end{claim}
\begin{claimproof}
Let $\{\varphi_i(x)\}_{i<k}$ be a family of p.p. formulas defining subgroups such that $\{\varphi_i(A)\}_{i<k}$ witnesses an {\rm inp}-pattern of depth $k$ in $A$, and suppose that each formula $\varphi_i(x)$ is  
$$
\bigwedge_{j} (n_{i,j}|m_{i,j}x)\wedge (m_ix=0)
$$
with $0 \le m_{i,j} < n_{i,j}$ and $0\le m_i$. Assume that $\{\varphi_i(x)\}_{i<k}$ does not yield an inp-pattern of depth $k$ in $A_{\Pp} \oplus A_\Y$. Thus, there exists some index $i_0<k$ such that 
$$
\left[ \bigcap_{i\neq i_0} \varphi_i(A_{\Pp}\oplus A_\Y) : \bigcap_i \varphi_i(A_{\Pp}\oplus A_\Y) \right] < \omega.
$$
Hence, the subgroup $\bigcap_{i} \varphi_{i}(A_\Y)$ has finite index in $\bigcap_{i\neq i_0} \varphi_i(A_\Y)$. Note that $m_j=0$ for all $j<k$. Indeed, if there is some $j\neq i_0$ such that $m_j>0$ then $\varphi_j(A_\X)=\{0\}$, yielding that $\bigcap_i \varphi_i(A)$ has finite index in $\bigcap_{i\neq i_0} \varphi_i(A)$, a contradiction. Thus, for all $j\neq i_0$ we have that  $m_j=0$ and so $\varphi_j(A_\Y)=A_\Y$, since $A_\Y$ is divisible. Hence, if $m_{i_0}>0$ then we get
\[
\left[ \bigcap_{i\neq i_0} \varphi_i(A_\Y) : \bigcap_i \varphi_i(A_\Y) \right] = \omega,
\] 
a contradiction. Consequently, each group $\varphi_i(A_\Y)$ has  unbounded exponent and we thus have that $\varphi_i(A_\Y)=A_\Y$. Therefore, the family $\{\varphi_i(x)\}_{i<k}$ witnesses an inp-pattern of depth $k$  in $A_{\Pp}\oplus A_\X$, as desired.
\end{claimproof}

To finish the proof, it remains to compute the $\dprk$-rank of the two distinct subgroups given by the claim. For this, it is convenient to partition the  set $\Pp$ in three subsets. Set
\begin{itemize}
\item $\mathcal P_1$ to be the set of primes $p$ such that $\alpha_{p,n}\neq 0$ for infinitely $n$,
\item $\mathcal P_2$ the set of primes $p$ such that $\alpha_{p,n} =0$ for all but finitely many $n$ with at least one $\alpha_{p,n}=\omega$, and
\item $\mathcal P_3$ the collection of primes $p$ such that $\alpha_{p,n}$ is finite for every $n$ and that $\alpha_{p,n}=0$ for all but finitely many $n$.
\end{itemize} 
Similarly as before, we set $A_{\Pp_i}$ to denote the subgroup
$$
\bigoplus_{p\in\Pp_i}\bigoplus_{n>0} \mathbb Z(p^n)^{(\alpha_{p,n})}.
$$
Note that $\Pp_1$ is $\mathrm{U}_{\ge\aleph_0}(A)$, the set of primes for which $A$ has unbounded length, and so it is finite since $A$ has finite dp-rank. Similarly, the set $\Pp_2$ is also finite but $\Pp_3$ could be infinite. Furthermore, by Propositions \ref{P:BoundedExponent} and \ref{P:Torsion+UnbddLength}, to get the equation from the statement we can clearly now assume that either $\X$ or $\Y$ are non-empty. In other words, either the subgroup $A_\X$ or the subgroup  $A_\Y$ is infinite. As a consequence, by Lemma \ref{L:UnbddExp+Dpmin} we have that
$$
\dprk(A) = \dprk( A_{\Pp_1}\oplus A_{\Pp_2} \oplus A_{\X} \oplus A_\Y ),
$$
since any p.p.-definable subgroup of $A_{\Pp_3}$ either has finite index if it has unbounded exponent or it is finite. 
Hence, the set $\Pp_3$ is negligible and so we may further assume that it is empty. Now, we may distinguish two cases.

\noindent{\bf Case 1}. We compute the $\dprk$-rank of $A_{\Pp_1} \oplus A_{\Pp_2} \oplus A_\X$, assuming that $\X$ is non-empty. It is convenient to denote by $\X_{\ge\aleph_0}$ the set of primes $p$ in $\X$ such that $\beta_p=\omega$ and $A_{\X_{\ge\aleph_0} }$ the subgroup of $A_\X$ consisting only of summands with $\beta_p=\omega$. 

If the sets $\Pp_1$ and $\X_{\ge\aleph_0}$ are both empty, then applying Lemma \ref{L:UnbddExp+Dpmin} we get that
$$
\dprk(A_{\Pp_1} \oplus A_{\Pp_2} \oplus A_\X) = \dprk(A_{\Pp_2})+1.
$$ 
Thus, we may suppose that at least one of these sets is non-empty, in which case again by Lemma \ref{L:UnbddExp+Dpmin} we obtain 
$$
\dprk(A_{\Pp_1} \oplus A_{\Pp_2} \oplus A_\X) = \dprk(A_{\Pp_1} \oplus A_{\Pp_2} \oplus A_{\X_{\ge\aleph_0}}).
$$ 
Under this assumption, we find a suitable inp-pattern witnessing that 
$$
\dprk(A_{\Pp_1} \oplus A_{\Pp_2} \oplus A_{\X_{\ge \aleph_0}}) =\dprk(A_{\Pp_1}) + \dprk(A_{\Pp_2}) + \dprk(A_{\X_{\ge\aleph_0}}).
$$ 
Before we proceed to the proof, we note that in particular this yields that
\[\dprk(A_{\Pp_1} \oplus A_{\Pp_2}) =\dprk(A_{\Pp_1}) + \dprk(A_{\Pp_2}).\]

We first consider two p.p-definable subgroups. Let $r$ be the exponent of $A_{\Pp_2}$, and note that 
$$
r (A_{\Pp_1}  \oplus A_{\Pp_2} \oplus A_{\X_{\ge \aleph_0}}) = A_{\Pp_1}\oplus \{0\} \oplus r A_{\X_{\ge \aleph_0}}$$
and
$$
(A_{\Pp_1} \oplus A_{\Pp_2} \oplus A_{\X_{\ge \aleph_0}}) [r] = \{0\}\oplus A_{\Pp_2}\oplus \{0\},
$$ 
where the former equality holds since $\Pp_1\cap \Pp_2= \emptyset$ and the latter equality holds since $A_\X$ is torsion-free and $A_{\Pp_1}$ has no $r$-torsion. 

Let $\{\chi_i(x)\}_{i<k_2}$ be a family of p.p. formulas witnessing an {\rm inp}-pattern in $A_{\Pp_2}$ with $k_2=\dprk(A_{\Pp_2})$ and assume, as we may, that each $\chi_i(x)$ defines a subgroup and implies $rx=0$. Consider the family $\{\psi_i(x)\}_{i<k_\X}$ of p.p. formulas witnessing that $\dprk(A_{\X_{\ge\aleph_0}})=k_\X$ in a way that each $\psi_i(x)$ defines in $A_{\X_{\ge\aleph_0}}$ a subgroup of the form  $pr|x$ for $p\in\X_{\ge\aleph_0}$, which is possible by  Proposition \ref{P:T-FStrong} applied to $A_{\X_{\ge\aleph_0}}$ and noticing that $A_{\X_{\ge\aleph_0}}$ and $rA_{\X_{\ge\aleph_0}}$ are isomorphic. In particular, it then follows that $\psi_i(A_{\Pp_1})=A_{\Pp_1}$, since $\mathcal P_1\cap\X = \emptyset$, and that $\psi_i(A_{\X_{\ge\aleph_0}})$ is contained in $rA_{\X_{\ge\aleph_0}}$.

Finally, let $\{\varphi_i(x)\}_{i<k_1}$ be a family of p.p. formulas witnessing an {\rm inp}-pattern in $A_{\Pp_1}$ with $k_1=\dprk(A_{\Pp_1})$. As usual, we may assume that each $\varphi_i(x)$ defines a subgroup and furthermore, by Proposition \ref{P:Torsion+UnbddLength} we can take all these formulas yielding subgroups of unbounded exponent. Since $\Pp_1$ is disjoint from $\Pp_2\cup \X$, we may assume that the primes appearing in these formulas are co-prime to any prime in $\Pp_2\cup\X$. Thus $\varphi_i(A_{\Pp_2}) = A_{\Pp_2}$ and $ \varphi_i(A_{\X_{\ge\aleph_0}})=A_{\X_{\ge\aleph_0}}$. 

Now, consider the family of p.p.-definable subgroups of $A_{\Pp_1} \oplus A_{\Pp_2} \oplus A_{\X_{\ge\aleph_0}}$ formed by:
\begin{gather*}
\varphi_i(A_{\Pp_1})  \oplus A_{\Pp_2} \oplus A_{\X_{\ge\aleph_0}}  \text{ \ for $i<k_1$, \ } A_{\Pp_1} \oplus \chi_i(A_{\Pp_2}) \oplus r A_{\X_{\ge\aleph_0}}  \text{ \ for $i<k_2$}  \\
\text{and \ } A_{\Pp_1} \oplus A_{\Pp_2} \oplus \psi_i(A_{\X_{\ge\aleph_0}}) \text{ \ for $i<k_\X$}.
\end{gather*}
By construction, it then follows that this family yields an inp-pattern of depth $k_1+k_2+k_\X$.

\noindent{\bf Case 2}. To calculate the $\dprk$-rank of $A_{\Pp_1} \oplus A_{\Pp_2} \oplus A_\Y$ assuming that $\Y$ is non-empty, similarly as before, it is convenient to denote by $\Y_{\ge\aleph_0}$ the set of primes $p$ in $\Y$ such that $\gamma_p=\omega$ and $A_{\Y_{\ge\aleph_0} }$ the subgroup of $A_\Y$ consisting only of summands with $\gamma_p=\omega$. 

If the sets $\Pp_1$ and $\Y_{\ge\aleph_0}$ are both empty, then applying Lemma \ref{L:UnbddExp+Dpmin} we get that
$$
\dprk(A_{\Pp_1} \oplus A_{\Pp_2} \oplus A_\Y) = \dprk(A_{\Pp_2})+1.
$$ 
Thus, we may suppose that at least one of these sets is non-empty, in which case again by Lemma \ref{L:UnbddExp+Dpmin} we obtain 
$$
\dprk(A_{\Pp_1} \oplus A_{\Pp_2} \oplus A_\Y) = \dprk(A_{\Pp_1} \oplus A_{\Pp_2} \oplus A_{\Y_{\ge\aleph_0}}).
$$ 
Under this assumption, we find a suitable inp-pattern witnessing that 
$$
\dprk(A_{\Pp_1} \oplus A_{\Pp_2} \oplus A_\Y) =\dprk(A_{\Pp_1}) + \dprk(A_{\Pp_2}) + \dprk(A_{\Y_{\ge\aleph_0}}).
$$ 
To do so, as before, set $r$ to be the exponent of $A_{\Pp_2}$ and note then that this time we have, since $\Pp_1\cap \Pp_2 =\emptyset$,
$$
r (A_{\Pp_1} \oplus A_{\Pp_2} \oplus A_{\Y_{\ge\aleph_0}}) = A_{\Pp_1} \oplus \{0\} \oplus A_{\Y_{\ge\aleph_0}}
$$
and 
$$(A_{\Pp_1} \oplus A_{\Pp_2} \oplus A_{\Y_{\ge\aleph_0}}) [r] = \{0\} \oplus  A_{\Pp_2} \oplus A_{\Y_{\ge\aleph_0}}[r],
$$
since {\it a priori} the sets $\Pp_2$ and $\Y_{\ge\aleph_0}$ may have elements in common. Let $\{\chi_i(x)\}_{i<k_2}$ be a family of p.p. formulas witnessing an {\rm inp}-pattern in $A_{\Pp_2}$ with $k_2=\dprk(A_{\Pp_2})$ and assume, as we may, that each $\chi_i(x)$ defines a subgroup and implies $rx=0$. Consider the family $\{\psi_i(x)\}_{i<k_\Y}$ of p.p. formulas witnessing that $\dprk(A_{\Y_{\ge\aleph_0}})=k_\Y$ in a way that each $\psi_i(x)$ defines in $A_{\Y_{\ge\aleph_0}}$ a subgroup of exponent $m_i>0$. More precisely, by Proposition \ref{P:Divisible+Torsion} we know that if $p_0,\ldots,p_{k_\Y -1}$ are all primes in ${\Y_{\ge\aleph_0}}$, then $m_i$ can be taken to be the product of all $p_j$ for $j\neq i$ and so 
$$
\psi_i(A_{\Y_{\ge\aleph_0}}) = \bigoplus_{j\neq i} \mathbb Z(p^{\infty})^{(\omega)}[p_j].
$$ 
In particular, it then follows that $\psi_i(A_{\Pp_1})=\{0\}$, since $\Pp_1\cap \Y=\emptyset$. Finally, let $\{\varphi_i(x)\}_{i<k_1}$ be a family of p.p. formulas witnessing an {\rm inp}-pattern in $A_{\Pp_1}$ with $k_1=\dprk(A_{\Pp_1})$. As usual, we may assume that each $\varphi_i(x)$ defines a subgroup and furthermore, by Proposition \ref{P:Torsion+UnbddLength} we can take all these formulas yielding subgroups of finite exponent. Since $\Pp_1$ is disjoint from $\Pp_2\cup \Y$, we may assume that these exponents are co-prime to any prime in $\Pp_2\cup\Y$. Thus $\varphi_i(A_{\Pp_2}) = \varphi_i(A_{\Y_{\ge\aleph_0}})=\{0\}$. 
Now, consider the family of p.p.-definable subgroups of $A_{\Pp_1} \oplus A_{\Pp_2} \oplus A_{\Y_{\ge\aleph_0}}$ formed by:
\begin{gather*}
\varphi_i(A_{\Pp_1})  \oplus A_{\Pp_2} \oplus \left( A_{\Y_{\ge\aleph_0}}[r] +\sum_{i<k_1} \psi_i(A_{\Y_{\ge\aleph_0}}) \right) \text{ \ for $i<k_1$, \ } \\ A_{\Pp_1} \oplus \chi_i(A_{\Pp_2}) \oplus A_{\Y_{\ge\aleph_0}}  \text{ \ for $i<k_2$ \ } 
\text{and \ }  \\ 
\left( \sum_{i<k_1}\varphi_i(A_{\Pp_1}) \right) \oplus A_{\Pp_2} \oplus \left( A_{\Y_{\ge\aleph_0}}[r] + \psi_i(A_{\Y_{\ge\aleph_0}}) \right) \text{ \ for $i<k_\Y$}.
\end{gather*}
By construction, it then follows that this family yields an inp-pattern of depth $k_1+k_2+k_\Y$.

Therefore, putting all this together we finish the proof. More precisely, combining these two cases with the previous claim, the equations from $(3)$ and $(4)$ are obtained by considering the cases when $\Pp_1$ is empty or not. This corresponds to, after assuming that $\delta = 0$, the group $A$ having finite $p$-length for every prime $p$ or not, respectively.
\end{proof}

To obtain Theorem \ref{T:Main} it suffices to apply the results of Section 4. Namely, given an abelian group $A$, the sets $\mathrm{Tf}_{\ge \aleph_0}(A), \mathrm{D}_{\ge\aleph_0}(A), \mathrm{U}_{\ge \aleph_0}(p;A)$ and $\mathrm{U}_{\ge\aleph_0}(A)$ are preserved under elementarily equivalence, and also is its $\dprk$-rank. Hence, Theorem \ref{T:Main} follows from the previous result using Szmielew's Theorem and Propositions \ref{P:T-FStrong}, \ref{P:Divisible+Torsion}, \ref{P:BoundedExponent} and \ref{P:Torsion+UnbddLength}.

\section{Abelian groups with additional structure}\label{s:last}

The situation is drastically distinct if we allow some extra structure. In this last section we provide an example of a divisible torsion-free abelian group with additional structure whose theory is $\omega$-stable but does not have finite {\rm dp}-rank, exemplifying the relevance of working in the pure language of groups along the paper. Furthermore, we finish the section and the paper by showing that infinite stable fields of finite $\dprk$-rank are algebraically closed and so tame model-theoretically.

\subsection{Abelian structure} Consider the $\dprk$-minimal group $\mathbb Q^{(\omega)}$, seen as $\bigoplus_{i\in \mathbb N} G_i$ where each $G_i$ is an isomorphic copy of $\mathbb Q$. 

Let $\mathcal P$ be the collection of all finite subsets of $\mathbb N$. For every finite, possibly empty, subset $I$ of $\mathbb{N}$, let $H_I$ be the subgroup of $G$ whose elements consist of $0$ everywhere except in the coordinates from $I$. Note that each $H_I$ is isomorphic to $\bigoplus_{i\in I} G_i$.  Now, consider the abelian structure $\mathcal G = (G,+,(H_I)_{I\in \mathcal P})$, which as remarked before is one-based.

\begin{proposition}
The structure $\mathcal G= (G,+,(H_I)_{I\in \mathcal P})$ is $\omega$-stable of Morley rank $\omega$ and does not have finite {\rm dp}-rank.
\end{proposition}
\begin{proof}
We first show that every p.p.-definable subgroup of $G$, in the given structure, is one of the $H_I$. For this, since $\mathcal G$ is a group with some additional structure, we obtain by \cite[Theorem A.1.1]{Hodges} that every p.p.-formula $\varphi(x_0,\dots,x_n)$ is equivalent to a formula of the form
$$
\exists \bar{y} \bigwedge_{I\in \mathcal P_0} \left( \sum_{i=0}^n \mu_{I,i} x_i +\sum_{i=1}^{|\bar y|} \lambda_{I,i} y_i \in H_I\right),
$$
where $\mathcal P_0$ is a finite subset of $\mathcal P$ and the coefficients $\mu_{I,i}$ and $\lambda_{I,i}$ are integers. Thus we would like to understand formulas of the form
$$
\exists \bar{y} \bigwedge_{I\in \mathcal P_0} \left( \mu_{I} x + b_I +\sum_{i=1}^{|\bar y|} \lambda_{I,i} y_i \in H_I\right),
$$
where each $b_I$ is an element of $\mathbb Q$.
Since the $H_I$ are pure subgroups ({\it i.e.} if whenever an element of $H_I$ has an $n$th root in $G$, it necessarily has an $n$th root in $H_I$), if we allow $\lambda_{I,i}\in \mathbb{Q}$ then the above formula becomes:
$$
\exists\bar{y} \bigwedge_{I\in \mathcal P_0} \left( x+b_I+\sum_{i=1}^{|\bar y|} \lambda_{I,i} y_i \in H_I \right).
$$
Furthermore, if the formula is consistent, then it defines a coset of the subgroup defined by the formula 
$$
\exists\bar{y} \bigwedge_{I\in \mathcal P_0} \left( x+\sum_{i=1}^{|\bar y|} \lambda_{I,i} y_i \in H_I \right).
$$

\begin{claim}
A subgroup definable by a formula of the form 
$$
\exists\bar{y} \bigwedge_{I\in \mathcal P_0} \left( x+\sum_{i=1}^{|\bar y|} \lambda_{I,i} y_i \in H_I\right)
$$
is either all of $G$ or one of the finite support subgroups.
\end{claim}
\begin{claimproof}
Suppose that $\mathcal P_0=\{I_1,\dots, I_n\}$ and $|\bar y|=m$. Denote by $x=(x_j)_{j\in\mathbb N}$ elements of $G$ and note that $(x_j)_{j\in\mathbb N}$ satisfies the given formula if and only if for every $j\in \mathbb{N}$ there are $h_{I_k,j}\in H_{I_k}$ with $h_{I_k,j}=0$ whenever $j\notin I_k$ such that the following equation has a solution
$$
\left(\begin{array}{c} h_{I_1,j} \\ \vdots \\ h_{I_n,j}\end{array}\right)  +      
\left(
\begin{array}{ccc}
\lambda_{I_1,1} &\dots & \lambda_{I_1,m} \\
\vdots & \ddots & \vdots \\
\lambda_{I_n,1} &\dots & \lambda_{I_n,m}
\end{array} \right) \left(\begin{array}{c} y_{1} \\ \vdots \\ y_{m} \end{array}\right) = 
\left(\begin{array}{c} x_{j}  \\ \vdots \\ x_{j}  \end{array}\right),
$$ 
 Thus, this system of equations in $x_j$ yields a subgroup of $\mathbb Q$ which is definable in the pure language of groups. Hence, it defines the trivial subgroup $\{0\}$ or $\mathbb{Q}$. 

Now, since the elements of $G$ have finite support, to argue that the given formula defines $G$ or one of the finite support subgroups, it suffices to analyze the above system of equations when $h_{I_k,j} = 0$ for every $k,j$. In that case, note that this system of equations on $x_j$ defines $\mathbb Q$ if and only if there is a non-zero element $x_j$ such that the column given by $x_j$ is in the column space of the matrix given by the coefficients $\lambda_{I_k,i}$, which do not depend on the coordinate $j$. Therefore, the above system of equations on $x_j$ defines the same subgroup for every index $j$ with $h_{I_k,j} = 0$, either $\{0\}$  or $\mathbb Q$. This yields the result.
\end{claimproof}

As a consequence, any definable subset of $G$ is a boolean combination of cosets of subgroups of finite support. Thus, any type is determined by the minimal coset where it is concentrated on. Hence, one can see that the theory of $\mathcal G$ has Morley rank $\omega$. Indeed,  the generic type has Morley rank $\omega$ and a type determined by the coset of $H_I$, say, has Morley rank $|I|$.

Finally, by considering a finite number of subgroups $H_{i_0},\ldots,H_{i_n}$, we obtain by Proposition \ref{P:inpattern-in-1based} that the subgroups $\bigoplus_{k\neq j} H_{i_k}$ for $j\le n$ give an {\rm inp}-pattern of depth $n+1$. Therefore, the structure $\mathcal G$ does not have finite {\rm dp}-rank.  
\end{proof}

\subsection{Fields} Concerning fields, it has been conjectured by Shelah that every strong stable field is algebraically closed. We now see that stable fields of finite $\dprk$-rank are algebraically closed.  Furthermore, it is worth noticing in contrast with Corollary \ref{C:TF} that in the pure language of fields the generic type does not control the dp-rank. For instance, in a separably closed field of infinite Ershov invariant the generic type has weight $1$ but the theory is not even strong. 

\begin{proposition}\label{P:Fields}
An infinite stable division ring of finite {\rm dp}-rank is an algebraically closed field.
\end{proposition}
\begin{proof} Let $(D,+,\times)$ be an infinite stable division ring and assume that it has finite {\rm dp}-rank. By a result of Cherlin and Shelah \cite{CherShe}, it suffices to show that it has finite {\rm U}-rank. To do so, fix some stationary complete type $p$ concentrated on $D$ of {\rm U}-rank one; for instance apply Zorn's Lemma to find a nonalgebraic type (over a model) all whose forking extensions are algebraic. Let $X$ be the set of its realizations, which is indecomposable subset of $D^+$ by \cite[Proposition 2.12]{ConPil}, and fix some element $a$ from $X$. Set $Y=X-a$ and note that $Y$ contains the identity element from $D^+$ and is also indecomposable. Furthermore, it has \rm{U}-rank one since the \rm{U}-rank is preserved under definable bijections.

Now, observe that any left multiplicative translate of $Y$ is indecomposable, since multiplication on the left yields an additive automorphism, and has also ${\rm U}$-rank one. In particular, any set of the form $b_1Y+\ldots + b_nY$ has finite ${\rm U}$-rank by the Lascar inequalities. In fact, since $b_1Y_1+\dots +b_nY_n$ is algebraic over $\bigcup_{i=1}^n b_iY_i$, by \cite[Proposition 2.5]{ConPil}, for any type $q$ concentrated on $b_1Y+\ldots + b_nY$ we have that $\mathrm{U}(q) = wt(q)$. As a result, the $\mathrm{U}$-rank of the set $b_1Y+\ldots + b_nY$ is bounded above by the $\dprk$-rank of $D$. 

By a suitable version Zilber's Indecomposable Theorem, see \cite[Fact 2.13]{ConPil}, the family of all $D^\times$-translates of $Y$ generates an infinite type-definable connected subgroup $H$ of $D^+$ which has finite ${\rm U}$-rank. In particular,  note that $H$ is $D^\times$-invariant and hence an ideal of $D$. Therefore, we have that $H=D$ and so $D$ has finite {\rm U}-rank, as desired. 
\end{proof}

\end{document}